\documentclass[a4paper,doublespace,11pt]{amsart}

\usepackage{amsmath}
\usepackage{amssymb}
\usepackage{verbatim}
\usepackage[all]{xy}
\usepackage{enumerate}
\usepackage{setspace}

\usepackage{color}

\newtheorem{thm}{Theorem}[section]
\newtheorem{prop}[thm]{Proposition}
\newtheorem{cor}[thm]{Corollary}
\newtheorem{lem}[thm]{Lemma}

\newtheorem{dfn}[thm]{Definition}

\newtheorem{rmk}[thm]{Remark}

\usepackage{amsfonts,amssymb,amscd,amsmath,latexsym,amsbsy,enumerate,stmaryrd,a4wide}
\numberwithin{equation}{section}
\newcommand{\cM}{\mathcal{M}}
\newcommand{\cH}{\mathcal{H}}

\newcommand{\id}{\textrm{id}}

\newcommand{\cB}{\mathcal{B}}
\newcommand{\wot}{\overline{\otimes}}

\newcommand{\bn}{\Bbb N}

\newcommand{\la}{\langle}
\newcommand{\ra}{\rangle}
\newcommand{\ot}{\otimes}

\newcommand{\nphi}{\mathfrak{n}_{\varphi}}
\newcommand{\npsi}{\mathfrak{n}_{\psi}}
\newcommand{\mphi}{\mathfrak{m}_{\varphi}}

\newcommand{\Aut}{\textrm{Aut}}
\newcommand{\cN}{\mathcal{N}}
\newcommand{\Kil}{\mathsf{K}}

\date{\noindent \today.
 }

\title[Haagerup property for arbitrary von Neumann algebras]{The Haagerup property for arbitrary von Neumann algebras}

\date{\noindent \today.
 }

\author{Martijn Caspers}
 \address{M. Caspers, Fachbereich Mathematik und Informatik der Universit\"at M\"unster,
Einsteinstrasse 62,
48149 M\"unster, Germany}
 \email{martijn.caspers@uni-muenster.de}
\thanks{MC is supported by the grant SFB 878 ``{\it Groups, geometry and actions}''}

\author{Adam Skalski}
\address{Institute of Mathematics of the Polish Academy of Sciences,
ul.~\'Sniadeckich 8, 00--956 Warszawa, Poland \newline \indent Faculty of Mathematics, Informatics and Mechanics, University of Warsaw, ul.~Banacha 2,
02-097 Warsaw, Poland}
\thanks{AS is partially supported by the Iuventus Plus grant IP2012 043872.}
\email{a.skalski@impan.pl}


\begin{document}

\begin{abstract}
We introduce a natural generalization of the Haagerup property of a finite von Neumann algebra to an arbitrary von Neumann algebra (with a separable predual) equipped with a normal, semi-finite, faithful weight and prove that this property  does not depend on the choice of the weight. In particular this defines the Haagerup property as an intrinsic invariant of the von Neumann algebra. We also show that such a generalized Haagerup property is preserved under taking crossed products by actions of amenable locally compact groups.

Our results are motivated by recent examples from the theory of discrete quantum groups, where the Haagerup property appears {\it a priori} only with respect to the Haar state.
\end{abstract}

\maketitle

\section{Introduction}

In \cite{HaaFree} Haagerup showed that the reduced group C$^\ast$-algebras of the free groups have the completely contractive approximation property (CCAP) (or the  metric approximation property in the sense of Grothendieck). In order to do so, he constructed a sequence of positive definite normalized $C_0$-functions $\{ \varphi_k \}_{k\in \bn}$ on the free group that converges to the identity function pointwise. The resulting property for arbitrary  groups, replacing pointwise convergence by uniform convergence on compacta, is nowadays commonly called the {\it Haagerup property}. In the literature this approximation property sometimes has alternative names, such as {\it Property H}  or {\it Gromov's a-T-menability}.  We refer to the book \cite{book} for further background.

\vspace{0.3cm}

The Haagerup property can  also be defined for a finite von Neumann algebra equipped with a fixed faithful tracial normal state, with the original motivation coming this time from the study of cocycle actions on von Neumann algebras \cite{CoJ}.  The following Defintion \ref{Dfn=HPTrace} is one of the equivalent ones \cite{Jol}, closest to the setup developed further in this article.

\begin{dfn}\label{Dfn=HPTrace}
Let $\cM$ be a finite von Neumann algebra and let $\tau$ be a faithful, normal tracial state on $\cM$. Then $(\cM, \tau)$ has the {\it Haagerup property} if there exists a sequence of normal, completely positive maps $\{\Phi_k: \cM \rightarrow \cM \}_{k \in \mathbb{N}}$ such that $\tau \circ \Phi_k \leq \tau$ and moreover the $L^2$-maps $x \Omega_\tau \mapsto \Phi_k(x) \Omega_\tau$ with $x \in \cM$ extend to compact maps converging to 1 strongly. Here $\Omega_\tau$ is the cyclic vector in the GNS-representation, so that $\tau(\: \cdot \:) = \langle \: \cdot \:\: \Omega_\tau,   \Omega_\tau \rangle$.
\end{dfn}

Consider the group von Neumann algebra of a discrete group equipped with the canonical tracial vector state given by the Dirac mass at the identity. Then, in \cite{Cho} it is proved that the (equivalently defined) von Neumann algebraic Haagerup property agrees with the Haagerup property of a discrete group. Later, the Haagerup property was considered in \cite{Bo93}, \cite{A-D}, and then studied  in detail in \cite{Jol}. In particular, Jolissaint \cite{Jol} proved that the Haagerup property is independent of the choice of the normal tracial state. He also showed that if $(\cM, \tau)$ has the Haagerup property, then the completely positive maps may be chosen unital and trace preserving, i.e. $\tau \circ \Phi_k = \tau$.

\vspace{0.3cm}

Further facts concerning the Haagerup property for finite von Neumann algebras can be found for example in \cite{JolMar} and \cite{BF}. In particular, \cite{BF},  in the same spirit as Jolissaint's result mentioned above, shows that the subtraciality condition can be  dropped completely in fact (but one still has to assume that the maps $\Phi_k$ admit bounded -- in fact, compact -- extensions to the relevant $L^2$-space). The paper \cite{BF} contains also a version of this result for correspondences.

\vspace{0.3cm}

Recent developments motivate the study of the Haagerup property beyond the finite case, see \cite{A-D-II}, \cite{HoudRic}, \cite{DawFimSkaWhi} and \cite{ComFreYam}. In \cite{DawFimSkaWhi} the Haagerup property was studied for locally compact quantum groups, where it appeared via replacing the normal, tracial state in the definition by the Haar state of a compact quantum group $\mathbb{G}$ (which is a faithful, normal, but in general non-tracial state on the von Neumann algebra $L^{\infty}(\mathbb{G})$). This naturally raises the question whether such a generalized Haagerup property is independent of the choice of state -- note in particular that for a general von Neumann algebra  there is no canonical choice for the state. In \cite{ComFreYam}, \cite{HoudRic}, \cite{A-D-II} concrete examples of  non-semi-finite von Neumann algebras with the Haagerup property understood as above were found. It is worth to note that similarly as the `finite' Haagerup property was motivated by the study of von Neumann algebras of discrete groups, the one for arbitrary von Neumann algebras appears naturally if one studies algebras associated to objects generalizing classical groups: quantum groups, as in \cite{DawFimSkaWhi}, or spaces of cosets, as in \cite{A-D-II}.

\vspace{0.3cm}

This article considers the following definition of the Haagerup property. We refer to Definition \ref{Dfn=HP} for details on the $L^2$-implementations.

\begin{dfn}\label{Dfn=HPIntro}
A pair $(\cM, \varphi)$ of a von Neumann algebra $\cM$ with a separable predual and normal, semi-finite, faithful weight $\varphi$ has the Haagerup property if there exists a sequence of normal, completely positive maps $\{ \Phi_k: \cM \rightarrow \cM \}_{k \in \mathbb{N}}$ such that $\varphi \circ \Phi_k \leq \varphi$ and such that their $L^2$-implementations are compact and converge to 1 strongly.
\end{dfn}

It deserves to be emphasized that Definition \ref{Dfn=HPIntro} agrees with the Haagerup property for finite von Neumann algebras in case $\varphi$ is a normal, tracial state. We summarize now the main results of this article. Firstly, we have the following theorem.

\begin{thm}\label{Thm=HPIntro}
Let $(\cM, \varphi)$ and $(\cM, \psi)$ be as in Definition \ref{Dfn=HPIntro}. $(\cM, \varphi)$ has the Haagerup property if and only if $(\cM, \psi)$ has the Haagerup property. That is, the Haagerup property is an intrinsic invariant of $\cM$.
\end{thm}

The strategy for our proof is as follows. We first prove Theorem \ref{Thm=HPIntro} for semi-finite von Neumann algebras. This can be done using Radon-Nikodym derivatives. For the general case, we rely on crossed product duality techniques, partially inspired by \cite{HaaKrau}. We prove that the Haagerup property is preserved under taking  the crossed product by the modular automorphism group of a weight and use this to derive weight independence from the semi-finite case. At this point it is essential to consider the Haagerup property with respect to  an arbitrary {\it weight}; it would have been insufficient to restrict Definition \ref{Dfn=HPIntro} to the case of a normal, faithful state. The reason is that given a state on a von Neumann algebra, its dual weight on the crossed product is not finite anymore. After deriving weight independence (Theorem \ref{Thm=HPIntro}) we are able to improve on our initial crossed product results and arrive at the following theorem.

\begin{thm}\label{Thm=HPCrossedProdIntro}
Let $\cM$ be a von Neumann algebra and let $\alpha: G \rightarrow \Aut(\cM)$ be any strongly continuous action of a second countable locally compact group $G$ on $\cM$ by automorphisms. If the crossed product $\cM\rtimes_\alpha G$ has the Haagerup property, then so does $\cM$.
\end{thm}

Finally, we find the converse of this result if $G$ is amenable.

\begin{thm}\label{Thm=ToCrossedProductIntro}
Let $G$ be amenable. If $\cM$ has the Haagerup property, then so does $\cM \rtimes_\alpha G$.
\end{thm}

As mentioned, our arguments are to an extent inspired by results of Haagerup and Kraus \cite{HaaKrau}. Compared to \cite{HaaKrau}, it is the interaction between the  $L^\infty$- and $L^2$-level appearing in the Haagerup property that requires overcoming extra difficulties.

\vspace{0.3cm}

An alternative approach to the Haagerup property for an arbitrary von Neumann algebra, based on working with  standard forms, has been  developed simultaneously with our work by Rui Okayasu and Reiji Tomatsu, see \cite{OkaTom}. The approaches turn out to be equivalent and the proof of this fact relies on the stability properties of the Haagerup property with respect to the crossed product, see \cite{COST}. After completion of both these works \cite{OkaTom}, \cite{COST} other (slightly different) approaches and proofs of these equivalences were found, see \cite{OkaTomII}, \cite{CasSkaII} and in particular \cite[Section 4]{CasSkaII} for a list of equivalences.

\vspace{0.3cm}

\noindent {\bf Contents.} The structure of this paper is as follows. Section \ref{Sect=Preliminaries} fixes preliminary notation and results. Section \ref{Sect=HP} defines the Haagerup property and proves some auxiliary lemmas. In Section \ref{Sect=SemiFinite} we treat the semi-finite case of Theorem \ref{Thm=HPIntro}. Section \ref{Sect=TypeIII} settles the notation for crossed products and contains the main results. Here we prove Theorems \ref{Thm=HPIntro} and \ref{Thm=HPCrossedProdIntro}. Finally, Section \ref{Sect=ToCrossedProduct} proves Theorem \ref{Thm=ToCrossedProductIntro}.

\section{Preliminaries}\label{Sect=Preliminaries}

In this section we fix some of the terminology and recall certain known results that will be of use in the following sections.

\vspace{0.3cm}

\noindent {\bf Convention:} {All von Neumann algebras $\cM$ in this article are assumed to have separable preduals. %, meaning that $\cM_\ast$ is a separable Banach space.
In particular, this implies that each $\cM$ admits a faithful, normal state. Groups are supposed to be second countable.}

\vspace{0.3cm}

% \vspace{0.3cm}

Given a von Neumann algebra $\cM$ with a faithful normal semi-finite weight $\varphi$ we write $\nphi$ for the left ideal $\{x\in \cM:\varphi(x^*x) <\infty\}$ and $L^2(\cM,\varphi)$ for the completion of $\nphi$ with respect to the scalar product $\la x, y \ra_{\varphi}= \varphi(y^*x)$, $(x,y \in \nphi)$. The GNS-embedding map from $\nphi$ into $L^2(\cM,\varphi)$ will be denoted by $\Lambda_\varphi$ or $\Lambda$ when it is clear which weight is involved, and the norm in $L^2(\cM,\varphi)$ by $\|\cdot\|_2$. Recall that $\Lambda_\varphi$ is a closed map with respect to $\sigma$-weak/weak topology. Furthermore, we write $\mphi$ for the linear span of $\nphi^* \nphi$, $(\sigma^{\varphi}_t)_{t\in \mathbb{R}}$ (or just $(\sigma_t)_{t\in \mathbb{R}}$) for the modular automorphism group of $\varphi$ and $J$ for the modular conjugation acting on $L^2(\cM,\varphi)$. The centralizer $\cM^{\varphi}$ is the von Neumann subalgebra of $\cM$ defined as $\{x\in \cM:\sigma^{\varphi}_t(x) =x, t \in \mathbb{R}\}$.  The $\sigma$-weak (ultraweak) tensor product of von Neumann algebras $\cM_1$ and $\cM_2$ will be denoted by $\cM_1 \wot \cM_2$. Generally speaking, we follow the notation of \cite{TakII}.

The following lemma is a standard application of the Kadison-Schwarz inequality.

\begin{lem}\label{Lem=KadisonSchwarz}
Let $\cM$ be a von Neumann algebra with two normal, semi-finite, faithful weights $\varphi$ and $\psi$. Let $\Phi: \cM \rightarrow \cM$ be a completely positive map such that $\varphi \circ \Phi \leq \psi$. Then, there exists a bounded map  $L^2(\cM, \psi) \rightarrow L^2(\cM, \varphi)$ determined by $\Lambda_\psi(x) \mapsto \Lambda_\varphi(\Phi(x)), x \in \npsi$. Its norm is not greater than $\Vert \Phi(1) \Vert^{\frac{1}{2}}$.
\end{lem}

\begin{comment}
\begin{proof}
Let $x \in \npsi$. Using the Kadison-Schwarz inequality $\Phi(x)^\ast \Phi(x) \leq \Vert \Phi(1) \Vert \Phi(x^\ast x)$, we see that $\Phi(x) \in \nphi$. Moreover,
\[
\Vert \Lambda_{\varphi}(\Phi(x)) \Vert_{2}^2 = \varphi(\Phi(x)^\ast \Phi(x)) \leq \Vert \Phi(1) \Vert \varphi(\Phi(x^\ast x)) \leq  \Vert \Phi(1) \Vert \psi(x^\ast x ) =  \Vert \Phi(1) \Vert \Vert \Lambda_{\psi}(x) \Vert_2^2.
\]
\end{proof}
\end{comment}

We will usually denote the induced map in $\cB(L^2(\cM, \psi) \rightarrow L^2(\cM, \varphi))$ by $T$. Following \cite{Jol} if $(\cM, \varphi)$ are as above (and $\psi = \varphi$)  we will say that a completely positive map $\Phi: \cM \rightarrow \cM$ such that $\varphi \circ \Phi \leq \varphi$ is $L^2(\cM,\varphi)$\emph{-compact} if the induced map $T\in \cB(L^2(\cM, \varphi))$ is compact.

The following Fubini type lemma allows us to interchange the integration with respect to a weight. Its abstract, generalized $C^*$-algebraic version can be found for example in \cite[Proposition 4.6.6]{Vaesthesis}.

\begin{lem}\label{Lem=Fubini}
Let $\cM$ be a von Neumann algebra with normal, semi-finite, faithful weight $\varphi$. Let $(X, \mu)$ be measure space with Radon measure $\mu$ on a Borel measure space $X$.  Let $f \in X \rightarrow \cM^+$ be a compactly supported $\sigma$-weakly integrable function. Then,
\[
\int_X \varphi(f(s)) d\mu(s) = \varphi\left( \int_X f(s) d\mu(s) \right).
\]
\end{lem}
\begin{proof}
Since $\varphi$ is normal, we may write \cite[Theorem VII.1.11]{TakII},
\[
\varphi(x) = \sup \left\{ \omega(x) \mid \omega \in A \right\}, \qquad {\rm with } \qquad
A = \left\{ \omega \in \cM_\ast^+ \mid \omega \leq \varphi \right\}.
\]
Then,
\[
\begin{split}
\varphi\left( \int_X f(s) d\mu(s) \right) = & \sup_{\omega \in A} \omega( \int_X f(s) d\mu(s) ) =
\sup_{\omega \in A} \int_X \omega(  f(s)  ) d\mu(s) \\
= & \int_X \sup_{\omega \in A}  \omega(  f(s)  ) d\mu(s) = \int_X \varphi(f(s)) d\mu(s).
\end{split}
\]
\end{proof}

The next result is a consequence of \cite[Lemma 3.18 (i)]{TakII}; we give a simple proof for the sake of completeness.

\begin{lem}
Let $\varphi$ be a normal, semi-finite, faithful weight on a von Neumann algebra $\cM$. Let $e \in \cM$ be in the centralizer of $\varphi$. Suppose that $\Vert e \Vert \leq 1$.  Then,
\begin{equation}\label{Eqn=CornerWeightInequality}
\varphi(e^\ast x e) \leq \varphi(x), \qquad \forall \: x \in \cM^+.
\end{equation}
\end{lem}
\begin{proof}
We may assume that $x \in \mphi^+$, since otherwise $\varphi(x) = \infty$ and this inequality is trivial. We may also assume that $x = y^\ast y$ for some $y \in \nphi$, since $\mphi^+$ is spanned by such elements. Then, using $\sigma_t^\varphi(e) = e$ in the second equality, and $J e^\ast e J \leq 1$ in the fourth (in)equality,
\[
\begin{split}
\varphi(e^\ast y^\ast y e) = & \langle  \Lambda_{\varphi}(y e),  \Lambda_{\varphi}(y e) \rangle = \langle J e J \Lambda_{\varphi}(y), J e J \Lambda_{\varphi}(y) \rangle \\
 = & \langle J e^\ast e J \Lambda_{\varphi}(y), \Lambda_{\varphi}(y) \rangle \leq \langle \Lambda_\varphi(y), \Lambda_{\varphi}(y) \rangle = \varphi(y^\ast y),
\end{split}
\]
which proves \eqref{Eqn=CornerWeightInequality}.
\end{proof}

Finally we shall briefly use the following standard lemma, which can be derived from Stone's theorem. We omit the proof.
\begin{lem}\label{Lem=StronglyCTGroup}
Let $\cM$ be a von Neumann algebra with normal, semi-finite, faithful weight $\varphi$. Let $\alpha: \mathbb{R} \rightarrow \Aut(\cM)$ be a strongly continuous $\varphi$-preserving 1-parameter group of automorphisms. Then, there exists a unique (unbounded) positive, self-adjoint operator $P$ such that $P^{it} \Lambda_{\varphi}(x) = \Lambda_{\varphi}(\alpha_t(x)), t \in \mathbb{R}$. The mapping $t \mapsto P^{it}$ is strongly continuous.
\end{lem}

\section{Definition of the Haagerup property}\label{Sect=HP}

We consider the most general definition of the Haagerup property with respect to a weight. In concrete examples, such as von Neumann algebras of discrete groups and discrete quantum groups, one is usually interested in the case that the weight is finite, i.e.\ a state. However, even if in the end one would want to consider only this restricted setup our proofs require the passage to the general, infinite context.

Recall Lemma \ref{Lem=KadisonSchwarz} and the terminology introduced after it.

\begin{dfn}\label{Dfn=HP}
A pair $(\cM, \varphi)$ of a von Neumann algebra $\cM$ with normal, semi-finite, faithful weight $\varphi$ has the Haagerup property if there exists a sequence $\{ \Phi_k: \cM \rightarrow \cM \}_{k \in \mathbb{N}}$ of normal, completely positive maps such that the following properties hold:
\begin{enumerate}
\item\label{Item=HPOne} for every $k \in \bn$ we have  $\varphi \circ \Phi_k \leq \varphi$;
\item\label{Item=HPTwo} for every $k \in \mathbb{N}$ the map $\Phi_k$ is $L^2(\cM, \varphi)$-compact;
\item\label{Item=HPThree} the induced maps $T_k\in \cB(L^2(\cM, \varphi))$ converge to $1$ in the strong operator topology of $\cB(L^2(\cM, \varphi))$.
\end{enumerate}
\end{dfn}
\begin{rmk}
Note that by the Principle of Uniform Boundedness, the {\it sequence} $\{ T_k \}_{k \in \mathbb{N}}$ of compact maps on $L^2(\cM, \varphi)$ is bounded. We use this remark implicitly in our proofs.
\end{rmk}

The above definition coincides with the definition introduced for a finite von Neumann algebra equipped with a faithful normal tracial state  in \cite{Cho} (see also the discussion in Section 4 of \cite{A-D}), and later studied in detail in \cite{Jol}. A recent definition of the Haagerup property proposed in \cite[Definition 6.3]{DawFimSkaWhi} (see also \cite{A-D-II} and \cite{ComFreYam}) assumes that the maps $\Phi_k$ are unital and state preserving. In fact \cite[ Proposition 2.2]{Jol} shows that if $\varphi$ is a faithful tracial state and $(\cM, \varphi)$ satisfies the Haagerup property, then the approximating maps $\Phi_k$ can all be chosen unital and trace preserving.

%We do not know if this remains true in the non-tracial or infinite context; this is in fact related to the question whether one can always choose $\Phi_k$ so that they commute with the modular group of $\varphi$.
% In case $(\cM, \varphi)$ is a pair of a von Neumann algebra equipped with a faithful, normal {\it state}, then we may indeed assume that the maps $\Phi_k$ satisfy these extra conditions as we shall prove in Theorem \ref{Thm=InvariantOfTheVNA}.

\begin{rmk}
Another possible route to defining a Haagerup-type property of a general von Neumann algebra, chosen for example in \cite{HoudRic}, is given via the \emph{compact approximation property} introduced in Definition 4.13 of \cite{A-D}: we say that a von Neumann algebra $\cM$ has the compact approximation property if there exists a sequence  $\{ \Phi_k: \cM \rightarrow \cM \}_{k \in \mathbb{N}}$ of normal, completely positive maps such that the following properties hold:
\begin{enumerate}
\item for every $k \in \mathbb{N}$, and $\xi \in L^2(\cM)$ the map $\cM \ni x \mapsto \Phi_k(x) \xi \in L^2(\cM)$ is compact (where $L^2(\cM)$ denotes the standard form Hilbert space of $\cM$);
\item for each $x \in \cM$ we have $\Phi_k (x) \stackrel{k\to \infty}{\longrightarrow} x$ in the $\sigma$-weak topology.
\end{enumerate}
It is not difficult to see that for a finite von Neumann algebra $\cM$ the Haagerup property implies the compact approximation property (in particular in the situation where $\varphi$ is a finite trace and $\Phi_i$ are unital the condition (2) above is equivalent to condition (3) in Definition \ref{Dfn=HP}); moreover they are equivalent for the group von Neumann algebras of discrete groups, as noted in Proposition 4.16 of \cite{A-D}.
\end{rmk}

Now we collect a couple of facts following immediately from the definition, which will be needed later on.

\begin{prop}\label{Prop=HPBofH}
Let $\cH$ be a separable Hilbert space. $(\cB(\cH), {\rm Tr})$ has the  Haagerup property.
\end{prop}
\begin{proof}
Let $\{ e_i \}_{i \in \mathbb{N}}$ be an orthonormal basis of $\cH$. Let $P_k$ be the projection onto the linear span of $\{ e_i \mid 0 \leq i \leq k \}$. Then we set,
\[
\Phi_k(x) =   P_{k} x P_{k}.
\]
It is straightforward to check that $\{ \Phi_k \}_{k \in \mathbb{N}}$ is a sequence of normal, completely positive maps that witnesses the  Haagerup property of $(\cB(\cH), {\rm Tr})$.
\end{proof}

\begin{comment}
\begin{rmk}
**** Possibly to be deleted or expanded *** Even in this basic case we do not know whether the approximating maps can be chosen to be unital and trace preserving. The problem is equivalent to answering the following, purely Hilbert space-theoretic question:  does there exist a family of vectors $\{\xi_{kl}:k,l \in \bn\}$ in a separable Hilbert space $\Kil$ satisfying the following conditions $(j,k,p,q \in \bn)$:
\[ \sum_{l=1}^{\infty} \la \xi_{kl}, \xi_{jl} \ra = \delta_{jk},\]
\[ \la \xi_{jk}, \xi_{pq} \ra = \la \xi_{qp}, \xi_{kj} \ra\]
and such that the operator  $T\in B(\ell^2 \ot \ell^2)$ given by
\[ \la e_k \ot e_m , T(e_p \ot e_q) \ra = \la \xi_{mq}, \xi_{kp} \ra,\]
is compact?
\end{rmk}
\end{comment}

Recall the tensor product construction for weights, as described for example in \cite[Section VIII.4]{TakII}.

\begin{lem}\label{Lem=HPTensor}
Let $(\cM, \varphi), (\cN, \psi)$ be pairs of a von Neumann algebra and a normal, semi-finite, faithful weight.  If $(\cM, \varphi)$ and $(\cN, \psi)$ have the Haagerup property, then so does $(\cM \wot \cN, \varphi \otimes \psi)$.
\end{lem}
\begin{proof}
If $\{ \Phi_k \}_{k \in \mathbb{N}}$ (resp. $\{ \Psi_k \}_{k \in \mathbb{N}}$) is a sequence of completely positive maps that witnesses the Haagerup property of $(\cM, \varphi)$ (resp. $(\cN, \psi)$), then $\{ \Phi_k \otimes \Psi_k \}_{k \in \mathbb{N}}$ witnesses the Haagerup property for $(\cM \wot \cN, \varphi \otimes \psi)$. We leave the details to the reader.
\end{proof}

\begin{rmk}
The converse of Lemma \ref{Lem=HPTensor} is also true. We found it more suitable to postpone the proof to Lemma \ref{Lem=DescentDownToTensor}.
\end{rmk}

\section{The Haagerup property for semi-finite von Neumann algebras}\label{Sect=SemiFinite}

This section is devoted to proving that for a semi-finite von Neumann algebra the Haagerup property is independent of the choice of weight. Our main tool is the Radon-Nikodym derivative \cite{PedTak}. Our proofs should be compared to those of Jolissaint from \cite{Jol}, treating the case of a finite von Neumann algebra. Here, we overcome some additional technicalities due to the fact that we work with infinite weights instead of finite tracial states.

\begin{lem}\label{Lem=HPCorner}
Let $\cM$ be a von Neumann algebra equipped with a normal, semi-finite, faithful weight $\varphi$. Let $\{e_n\}_{n \in \mathbb{N}}$ be a sequence of projections contained in the centralizer of $\varphi$ such that $e_n \nearrow 1$. Let $\varphi^{(n)}$ be the restriction of $\varphi$ to $e_n  \cM e_n$.  If $(\cM, \varphi)$ has the Haagerup property then for every $n \in \mathbb{N}$, $(e_n \cM e_n, \varphi^{(n)})$ has the Haagerup property. Conversely, if  for every $n \in \mathbb{N}$, $(e_n \cM e_n, \varphi^{(n)})$ has the Haagerup property and the approximating maps (i.e. $\Phi_k$ in Definition \ref{Dfn=HP}) are contractive for each $n$, then $(\cM, \varphi)$ has the Haagerup property.
\end{lem}
\begin{proof}
It deserves to be emphasized that for $x \in \mphi$ and $a$ in the centralizer of $\varphi$, the elements $xa$ and $ax$ are in $\mphi$, see \cite[Lemma VIII.2.4]{TakII}. So indeed, the restriction of $\varphi$ to the corner $e_n \cM e_n$ is semi-finite. The GNS map of $\varphi^{(n)}$ denoted by $\Lambda_{\varphi^{(n)}}$ can be taken as the restriction of $\Lambda_{\varphi}$ to $\nphi \cap e_n \cM e_n$.

First, assume that $(\cM, \varphi)$ has the Haagerup property.
Let $\{ \Phi_k \}_{k \in \mathbb{N}}$ be the sequence of normal, completely positive maps satisfying conditions \ref{Item=HPOne} - \ref{Item=HPThree} of Definition \ref{Dfn=HP}, and hence witnessing that $(\cM, \varphi)$ has the Haagerup property. Let $T_k$ be the corresponding $L^2$-map of $\Phi_k$, see Lemma \ref{Lem=KadisonSchwarz}.   Fix an index $n \in \mathbb{N}$. Set $\Psi_k(\: \cdot \:) = e_n \Phi_k( \: \cdot \:) e_n$. We claim that the sequence $\{ \Psi_k\}_{k \in \mathbb{N}}$ witnesses the Haagerup property for $(e_n \cM e_n, \varphi^{(n)})$. Using \eqref{Eqn=CornerWeightInequality} it follows directly that Property \ref{Item=HPOne} of Definition \ref{Dfn=HP} is satisfied. For $x \in \mathfrak{n}_{\varphi^{(n)}}$ it follows again from \cite[Lemma VIII.2.4]{TakII} that,
\[
\Psi_k(x)^\ast \Psi_k(x) = e_n \Phi_k(x)^\ast e_n \Phi_k(x) e_n \leq e_n \Phi_k(x)^\ast \Phi_k(x) e_n \in \mathfrak{m}_{\varphi^{(n)}},
\]
which means that $\Psi_k(x) \in \mathfrak{n}_{\varphi^{(n)}}$.  Furthermore, using  that $\sigma_t^{\varphi}(e_n) = e_n$, we have that,
\[
\Lambda_{\varphi^{(n)}}( x ) \mapsto \Lambda_{\varphi^{(n)}}( \Psi_k(x)) = \Lambda_{\varphi^{(n)}}( e_n \Phi_k(x) e_n ) = e_n J e_n J T_k \Lambda_{\varphi}( x ), \qquad x \in \mathfrak{n}_{\varphi^{(n)}},
\]
extends to a compact map $T_k^{(n)}$. Moreover, for $x \in \mathfrak{n}_{\varphi^{(n)}}$, we have,
\[
\begin{split}
T_k^{(n)}\Lambda_{\varphi^{(n)}}( x) = & \: \Lambda_{\varphi^{(n)}}( \Psi_k(x)) =
e_n J e_n J T_k \Lambda_{\varphi}( x )  \\
\rightarrow &  \: e_n J e_n J \Lambda_{\varphi}( x ) =  \Lambda_{\varphi^{(n)}}( e_n x e_n ) =  \Lambda_{\varphi^{(n)}}( x ),
\end{split}
\]
where the convergence is in norm of $L^2(\cM, \varphi)$. Since the norms $\Vert T_k \Vert$ are uniformly bounded, it follows from a $3\epsilon$-estimate that for every $n\in \mathbb{N}$ we have $T_k^{(n)} \stackrel{k\to \infty}{\longrightarrow}1$ strongly.

\vspace{0.3cm}

Next, we prove the converse.   Let $\{ \Phi_k^{(n)} \}_{k \in \mathbb{N}}$ be the sequence of completely positive maps witnessing the Haagerup property for $(e_n \cM e_n, \varphi^{(n)})$. By assumption we may assume that $\Phi_k^{(n)}$ is contractive. Also, let $T_k^{(n)}$ be their corresponding $L^2$-implementations, see Lemma \ref{Lem=KadisonSchwarz}. By the Kadison-Schwarz inequality also $T_k^{(n)}$ is contractive.  Let $F$ be a finite subset of $\nphi$. Choose $n \in \mathbb{N}$ such that for every $x \in F$ we have
\[
\Vert \Lambda_{\varphi}(e_n xe_n) - \Lambda_{\varphi}(x) \Vert_2 = \Vert e_n J e_n J \Lambda_{\varphi}(x) - \Lambda_{\varphi}(x) \Vert_2 \leq \frac{1}{n}.
\]
Then, choose $k := k(n)$ such that for every $x \in F$ we have
\[
\Vert T_k^{(n)}  \Lambda_{\varphi}(e_n x e_n) - \Lambda_{\varphi}(e_n xe_n) \Vert_2 \leq \frac{1}{n}.
\]
We claim that $\Psi_n(x) := \Phi_{k(n)}^{(n)}(e_n x e_n), x \in \cM$ witnesses the Haagerup property for $(\cM, \varphi)$. Indeed, using \eqref{Eqn=CornerWeightInequality} one verifies property \ref{Item=HPOne} of Definition \ref{Dfn=HP}. Let $x \in \nphi$. Then, $e_n x^\ast e_n x e_n \leq e_n x^\ast x e_n \in \mphi$ by \cite[Lemma VIII.2.4]{TakII}, so that $e_n x e_n \in \nphi$. Then, also $\Psi_n(x) \in \nphi$. Moreover,
\[
\Lambda_\varphi( \Psi_n(x) ) = T_{k(n)}^{(n)} e_n J e_n J \Lambda_{\varphi}(x),
\]
so that $\Psi_n$ determines a compact operator $\Lambda_\varphi(x) \mapsto \Lambda_\varphi(\Psi_n(x))$. Moreover, for every $x \in F$, we have,
\[
\begin{split}
\Vert \Lambda_{\varphi}(\Psi_n(x )  ) - \Lambda_{\varphi}(x) \Vert_{2}
= & \Vert T_{k(n)}^{(n)}   \Lambda_{\varphi}(  e_n xe_n   ) - \Lambda_{\varphi}(x) \Vert_{2} \\
\leq & \Vert T_{k(n)}^{(n)}  \Lambda_{\varphi}( e_n x e_n   ) - \Lambda_{\varphi}(e_n x e_n) \Vert_2 + \Vert \Lambda_{\varphi}(e_n xe_n) -  \Lambda_{\varphi}(x) \Vert_{2}
\leq \frac{2}{n}.
\end{split}
\]
Now, using the separability of $L^2(\cM, \varphi)$ and letting $n$ depend on the finite set $F$, this implies that we may in fact find a suitable sequence $\Psi_n$ such that properties \ref{Item=HPTwo} and \ref{Item=HPThree} are satisfied. Note that the $L^2$-implementations of $\Psi_n$ converge to 1 strongly by the fact that they are contractive, the previous equations and a standard $3\epsilon$-estimate.
\end{proof}

%\begin{lem}\label{Lem=MainResult}
%Let $\cM$ be a  von Neumann algebra equipped with normal faithful tracial state $\tau$. Let $h \in \cM^+$ be boundedly invertible and let $\varphi(\: \cdot \:) =  \tau(h\: \cdot \: h)$. If $(\cM, \varphi)$ has the symmetric Haagerup property, then the completely positive maps $\Phi_k, k \in \mathbb{N}$ in Definition \ref{Dfn=HAPSymmetric} may be chosen contractive.
%\end{lem}

In the formulation of the next proposition we use the Radon-Nikodym theorem for weights on von Neumann algebras, \cite[Theorem 5.12]{PedTak} (see also \cite[Corollary VIII.3.6]{TakII}).

\begin{prop}\label{Prop=CocycleInvariance}
Let $\cM$ be a semi-finite von Neumann algebra with normal, semi-finite, faithful trace $\tau$ and let $\varphi$ be a normal, semi-finite, faithful weight on $\cM$.  $(\cM, \tau)$ has the Haagerup property if and only if so does $(\cM, \varphi)$.
\end{prop}
\begin{proof}
Let $h$ be the Radon-Nikodym derivative of $\varphi$ with respect to $\tau$, i.e.\ the self-adjoint operator affiliated with $\cM$ such that $\varphi = \tau_h$ (see \cite[Lemma VIII.2.8]{TakII}  for the meaning of the last formula). First we treat the case that $h$ is a bounded and boundedly invertible operator. In this case, we have $\varphi(x ) = \tau(h^{\frac{1}{2}} x h^{\frac{1}{2}})$ for every $x \in \cM^+$.  Let $\{ \Psi_k \}_{k \in \mathbb{N}}$ be the sequence of completely positive maps that witnesses that $(\cM, \tau)$ has the Haagerup property, i.e.\ satisfies conditions \ref{Item=HPOne} -- \ref{Item=HPThree} of Definition \ref{Dfn=HP} with $\varphi$ replaced by $\tau$. In particular, we let $T_k$ be the compact operator on $L^2(\cM, \tau)$ determined by $\Lambda_{\tau}(x) \mapsto \Lambda_{\tau}( \Psi_k(x)) $. Also, $T_k \rightarrow 1$ strongly.

Now, define $\Phi_k$ by setting,
\[
\Phi_k( x) = h^{-\frac{1}{2}} \Psi_k(h^{\frac{1}{2}}  x h^{\frac{1}{2}} )h^{-\frac{1}{2}}, \quad x \in \cM.
\]
Clearly, $\Phi_k$ is a completely positive map. Now, for $x \in \cM^+$ we have,
\[
\varphi(\Phi_k(x)) = \tau(\Psi_k (h^{\frac{1}{2}} x h^{\frac{1}{2}} ) ) \leq \tau(h^{\frac{1}{2}} x h^{\frac{1}{2}} ) = \varphi(x),
\]
and hence \ref{Item=HPOne} of Definition \ref{Dfn=HP} follows. For the GNS-map we have $\Lambda_{\varphi}(x) = \Lambda_{\tau}(x h^{\frac{1}{2}})$, where $x \in \nphi = \mathfrak{n}_{\tau} h^{-\frac{1}{2}}$. Recall that $\nphi$ and $ \mathfrak{n}_{\tau}$ are left ideals. Take $x \in \nphi$.  Then, $h^{\frac{1}{2}} x h^{\frac{1}{2}} \in  \mathfrak{n}_{\tau}$, hence $\Psi_k(h^{\frac{1}{2}} x h^{\frac{1}{2}} ) \in  \mathfrak{n}_{\tau}$, and hence $\Phi_k( x) = h^{-\frac{1}{2}} \Psi_k(h^{\frac{1}{2}}  x h^{\frac{1}{2}} )h^{-\frac{1}{2}} \in  \mathfrak{n}_{\varphi}$. Moreover,
\[
\begin{split}
\Lambda_{\varphi}(\Phi_k(x)) = & \Lambda_{\tau}( \Phi_k(x) h^{\frac{1}{2}} ) = h^{-\frac{1}{2}} \Lambda_{\tau} ( \Psi_k(h^{\frac{1}{2}} x h^{\frac{1}{2}} ) ) \\
= & h^{-\frac{1}{2}} T_k h^{\frac{1}{2}} \Lambda_{\tau}( x h^{\frac{1}{2}}) =  h^{-\frac{1}{2}} T_k h^{\frac{1}{2}} \Lambda_{\varphi}( x).
\end{split}
\]
It is clear that the mappings $\Lambda_{\varphi}(x) \mapsto \Lambda_{\varphi}(\Phi_k(x))$ extend to compact operators and as $k \rightarrow \infty$ tend strongly to 1.

Now we pass to the the general case. Let $h$ be as in the statement of the proposition and let $e_n$ be the spectral projection of $h$
for the interval $[\frac{1}{n}, n]$. Consider the corner algebra $e_n \cM e_n$. Both $\varphi$ and $\tau$ restrict to (normal, faithful) semi-finite weights on $e_n \cM e_n$ (see the first paragraph of the proof of Lemma \ref{Lem=HPCorner}). If $\varphi^{(n)}$ and $\tau^{(n)}$ denote their respective restrictions, then  $\varphi^{(n)} = \tau^{(n)}_{e_n h}$. Note that $e_n h$ is invertible in $e_n \cM e_n$. So, if $(\cM, \tau)$ has the Haagerup property, then  we see that $(e_n \cM e_n, \tau^{(n)} )$ has the Haagerup property for every $n$, c.f. Lemma \ref{Lem=HPCorner}. By the earlier paragraphs this implies that $(e_n \cM e_n, \varphi^{(n)} )$ has the Haagerup property for every $n\in \mathbb{N}$.  Moreover, \cite[Lemma 3.3]{CasSkaII} which also holds without the word `symmetric' in its statement, the proof being essentially the same and independent of any other results, shows that the approximating maps witnessing the Haagerup property for each $(e_n \cM e_n, \varphi^{(n)}), n \in \mathbb{N}$ may be chosen contractive \footnote{This is necessary to assure the uniform bounds on the $L^2$-implementations; the original arXiv version of this article had a gap here, we thank the referee and Narutaka Ozawa for pointing this out.}.  Applying Lemma \ref{Lem=HPCorner} once more we see that $(\cM, \varphi)$ has the Haagerup property.

\end{proof}

\begin{thm}\label{Thm=HPSemiFinite}
Let $\cM$ be a semi-finite von Neumann algebra. Let $\varphi$ and $\psi$ be normal, semi-finite, faithful weights on $\cM$. Then $(\cM, \varphi)$ has the Haagerup property if and only if $(\cM, \psi)$ has the Haagerup property. In particular the Haagerup property is an invariant of $\cM$.
\end{thm}
\begin{proof}
Without loss of generality, we may assume that $\varphi = \tau$ is a normal, semi-finite, faithful trace. But then, the theorem follows from Proposition \ref{Prop=CocycleInvariance}, since every normal, semi-finite, faithful weight $\psi$ is of the form $\varphi_h$ for some self-adjoint $h$ affiliated with $\cM$. See \cite[Theorem VIII.3.14]{TakII} for details.
\end{proof}

\begin{comment}%{\color{blue}
Note in fact that following the proof of Theorem \ref{Thm=HPSemiFinite} shows that if $T_k$ are the approximating $L^2$-implementations of $(\cM, \varphi)$ (see Definition \ref{Dfn=HP}) then the approximating $L^2$-implementations of $(\cM, \varphi)$ may be chosen such that $\Vert S_k \Vert \leq \sup_k \Vert T_k \Vert$. In fact from the tracial case we may assume that $\sup_k \Vert T_k \Vert = 1$ see also \cite{CasSkaII}.
%}
\end{comment}

\section{The Haagerup property for arbitrary von Neumann algebras}\label{Sect=TypeIII}

In this section we treat general  von Neumann algebras (with separable preduals) and prove one of the main results of this article: the Haagerup property of a von Neumann algebra does not depend on the choice of the weight. Our main tool is crossed product duality: we show that the Haagerup property is preserved under passage to the crossed product by an action satisfying suitable conditions. Our proofs should to some extent be compared to those of Haagerup and Kraus \cite{HaaKrau}.
%In this context, let us also mention \cite{HaaJunXu}, where extensions of maps on $\cM$ to the crossed product were investigated. We shall not use the latter results.

\subsection{Preliminaries regarding von Neumann algebraic crossed products} \label{Subsection:crossprel}

We recall the necessary preliminaries on crossed products, for which we refer to \cite{TakII}, \cite{Tak73}, \cite{HaaMathScanI} and  \cite{HaaMathScanII} (we follow the notation and terminology of \cite{HaaKrau}). Let $\cM$ be a von Neumann algebra equipped with a normal, semi-finite, faithful weight $\varphi$. The Hilbert space on which $\cM$ acts will be denoted by $\cH$. Furthermore, let $G$ be a locally compact group (with a fixed left invariant Haar measure, the resulting Hilbert space $L^2G$ and the modular function $\Delta_G$; the norm in $L^1 G$ will be denoted by $\Vert \cdot \Vert_1$) and $\alpha: G \rightarrow \Aut(\cM)$ be a strongly continuous group of automorphisms of $\cM$. The crossed product von Neumann algebra $\cN := \cM \rtimes_\alpha G$ is the von Neumann algebra acting on $\cH \otimes L^2G \simeq L^2(G, \cH)$ generated by the operators
\[
\pi(x), \;\;  x \in \cM, \qquad \textrm{ and } \qquad \lambda(t),  \;\; t \in G,
\]
that are determined by the formulas:
\[
\begin{array}{rll}
(\pi(x)\xi)(s) =& \alpha_{s}^{-1}(x) \xi(s),& \qquad s \in G, \\
(\lambda(t) \xi)(s) =& \xi(t^{-1}s),& \qquad s \in G.
\end{array}
\]

We shall work with the following useful presentation for the crossed product (see \cite{HaaMathScanI}, \cite{HaaMathScanII}). Let $C_c(G, \cM)$ denote the space of compactly supported $\sigma$-strong-$\ast$ continuous functions on $G$ with values in $\cM$. We endow $C_c(G, \cM)$ with an involution and a convolution product as follows: for $x, y \in C_c(G, \cM)$ we set
\[
(x \ast y)(s) = \int_G \alpha_t(x(st) ) y(t^{-1}) dt, \qquad x^\sharp(t) = \Delta_G(t)^{-1} \alpha_t^{-1}(x(t^{-1})^\ast).
\]
Furthermore, we set
\[
\mu(x) = \int_G \lambda(t) \pi(x(t)) dt.
\]
Then $\mu$ defines an involutive representation of $C_c(G, \cM)$ on $L^2(G, \cH)$ which maps into the crossed product $\cN$. For $f \in C_c(G)$ we define $x(f) \in C_c(G, \cM)$ by the formula $x(f)(s) = f(s) 1_{\cM}$. Then
\[
\mu(x(f)) = \lambda(f).
\]
We shall also use the following formula. Fix $x \in \cM$ and $f \in C_c(G)$. Put $a(s) = f(s) \alpha_s^{-1}(x), s \in G$, so that $a \in C_c(G, \cM)$. Then,
\begin{equation}\label{Eqn=MuEmbedding}
\mu(a) = \int_G \lambda(s) \pi(a(s)) ds =   \int_G  \lambda(s) f(s) \pi( \alpha_s^{-1}(x)) ds = \int f(s) \pi(x) \lambda(s) ds = \pi(x) \lambda(f).
\end{equation}

Recall the extended positive part $\cM^+_{{\rm ext}}$ of $\cM$, as defined in \cite[Definition IX.4.4]{TakII}. We refer to \cite[Section IX.4]{TakII} for the theory of operator valued weights (see also \cite{HaaOVW}). Let $T: \cN^+ \rightarrow \cM^+_{{\rm ext}}$ be the normal, semi-finite, faithful operator valued weight constructed in   \cite[Theorem 3.1]{HaaMathScanII}. By \cite[Theorem 3.1 (c)]{HaaMathScanII} it satisfies the property
\begin{equation}\label{Eqn=TWeight}
T(\mu( x^\sharp \ast x) ) = \pi((x^\sharp \ast x)(e)), \qquad x \in C_c(G, \cM).
\end{equation}

In fact the articles \cite{HaaMathScanI}, \cite{HaaMathScanII} contain the following result, concerning the core of the GNS map for the \emph{dual weight} $\tilde{\varphi}$.

\begin{lem}\label{Lem=CoreLemma}
Let $\varphi$ be a normal, semi-finite, faithful weight $\varphi$ on $\cM$ and set $\tilde{\varphi} = \varphi \circ \pi^{-1}  \circ T$. Then, $\tilde{\varphi}$ is a normal, semi-finite, faithful weight on $\cN$. Let,
\[
B_\varphi = {\rm span } \: \left\{ a \cdot x \mid a \in C_c(G, \cM), x \in \mathfrak{n}_\varphi \right\}.
\]
 The set $\mu(B_\varphi)$ is contained in $\mathfrak{n}_{\tilde{\varphi}}$ and is a $\sigma$-weak/norm core for the GNS-map $\Lambda_{\tilde{\varphi}}$.
\end{lem}
\begin{proof}
The statement about the core follows from \cite[Definition 3.1]{HaaMathScanI}. We refer the reader to \cite[Theorem 3.1 (d)]{HaaMathScanII} to see that the definition of $\tilde{\varphi}$ given in  \cite[Definition 3.1]{HaaMathScanI} indeed agrees with the one adopted above: $\tilde{\varphi} = \varphi \circ \pi^{-1} \circ T$.
\end{proof}

Using \eqref{Eqn=TWeight} one may prove the following statement.

\begin{lem}[c.f. Lemma 3.5 of  \cite{HaaKrau}]\label{Lem=HaagerupKrausComputation}
Let $\cM, \cN, \alpha,  T$ be as above. Let $f \in L^2(G)$ be such that $g \mapsto f \ast g, g \in C_c(G)$ extends to a bounded operator $\lambda(f)$ on $L^2(G)$. Then $\lambda(f) \in \mathfrak{n}_T$ and the map $T_f: \cN \rightarrow \cM$ defined via the prescription
\[
T_f(y) = T(\lambda(f)^\ast y \lambda(f)), \;\;\; y \in \cN,
\]
is a normal, completely positive map. Moreover, for $x \in \cM$ we have
\begin{equation}\label{Eqn=Tf}
T_{f}(\pi(x)) = \pi\left( \int_{G} \vert f(s) \vert^2 \alpha_{s}^{-1}(x) ds \right).
\end{equation}
\end{lem}
\begin{proof}
By \cite[Lemma 3.5]{HaaKrau} the statement is true for $f \in C_c(G)$. Recall that $\lambda(C_c(G))$ is a $\sigma$-weak/norm core for the GNS-map of the Plancherel weight on the group von Neumann algebra of $G$ \cite[Definition VII.3.2]{TakII}. So in case of a general $f \in L^2(G)$ as above, we may take a net $f_j \in C_c(G)$ such that $\lambda(f_j) \rightarrow \lambda(f)$ $\sigma$-weakly and $\Vert f_j - f \Vert_{L^2(G)} \rightarrow 0$. Let $\omega \in \cM_\ast$. Using polarization \eqref{Eqn=Tf} implies,
\[
\omega \circ T(\lambda(f_i)^\ast   \lambda(f_j) ) =    \omega(1) \langle f_j, f_i \rangle_{L^2(G), L^2(G)}.
\]
It follows that $\Lambda_{\omega \circ T}(\lambda(f_j))$ is a Cauchy net as $f_j$ is a Cauchy net. Since $\Lambda_{\omega \circ T}$ is $\sigma$-weak/norm closed it follows that $\lambda(f) \in \mathfrak{n}_{\omega \circ T}$  for every $\omega \in \cM_\ast$. This implies that $T(\lambda(f)^\ast \lambda(f)) \in \cM^+$ and so $\lambda(f) \in \mathfrak{n}_T$. Then $T_f$ is a normal completely positive map. Finally, let again $\omega \in \cM_\ast$. Since $\alpha$ is strongly continuous $s \mapsto \omega(\alpha_s(x))$ is continuous for every $x \in \cM$ \cite[Theorem II.2.6]{Tak1}. Then,
\[
\begin{split}
&\omega \circ T_{f}( \pi(x) ) = \langle x \Lambda_{\omega \circ T}(\lambda(f)), \Lambda_{\omega \circ T}(\lambda(f)) \rangle \\
=&  \lim_j \langle x \Lambda_{\omega \circ T}(\lambda(f_j)), \Lambda_{\omega \circ T}(\lambda(f_j)) \rangle =
\lim_j \int_{G} \vert f_j(s) \vert^2 \omega(\alpha_{s}^{-1}(x)) ds \\
= & \int_{G} \vert f(s) \vert^2 \omega(\alpha_{s}^{-1}(x)) ds.
\end{split}
\]
This completes the proof.
\end{proof}
Next, we recall the constructions dual to those we have introduced so far.
The von Neumann algebra $\cN$ can naturally be viewed as a von Neumann subalgebra of the ultraweak tensor product $\cM \wot B(L^2 G)$. Explicitly the embedding is given as follows: the group von Neumann algebra $\mathcal{L}(G) \subseteq \cN$ is embedded inside $\cM \wot B(L^2 G)$ by $\lambda(t) \mapsto 1_{\cM} \otimes \lambda(t)$, where $t \in G$, and each operator $\pi(x)$ ($x \in \cM$) can be identified with the operator $(\pi(x)\xi)(s) = \alpha_s^{-1}(x) \xi(s)$, where $\xi \in L^2(G, \cH)$. For $s \in G$ and $f \in L^2G$ define
\[
(r_s f)(t) = \Delta_G(s)^{\frac{1}{2}} f(ts), \qquad \rho_s = {\rm Ad}\: r_s \qquad {\rm and} \qquad \beta_s = \alpha_s \otimes \rho_s,
\]
so that $\rho$ and $\beta$ are strongly continuous  groups of automorphisms acting on respectively $\cB(L^2 G)$ and $\cM \wot \cB(L^2 G)$. Then $\cN$ equals exactly the fixed point algebra of $\beta$. Moreover, there exists a normal, semi-finite, faithful operator valued weight $S$:
\[
S: (\cM \wot  \cB(L^2 G))^+ \rightarrow \cN^+_{{\rm ext}}:\;\; x \mapsto \int_G \beta_s(x) ds.
\]
If $\psi$ is a weight on $\cN$, then the composition $\psi \circ S$ is a weight on $\cM \wot \cB(L^2 G)$,  preserved under the action of $\beta$.  We also set, for $t \in G, f\in C_c(G)$,
\[
_tf(s) = f(ts), \qquad { f_t(s) = f(st), \qquad f^\ast(s) = \overline{f(s^{-1})}\Delta_G(s^{-1}), \qquad s \in G,}
\]
so that for example $\lambda(f^\ast) = \lambda(f)^\ast$. Finally define (for each $f \in L^{\infty}(G)$) $\nu(f) = 1_{\cM} \otimes f \in \cM \wot \cB(L^2G)$, where the right leg is understood as a multiplication operator. One may then verify that
\begin{equation}\label{Eqn=BetaNu}
\beta_t(\nu(f)) = \nu(f_t), \qquad t \in G.
\end{equation}

\vspace{0.3cm}

In the special case of $G = \mathbb{R}$, the operator valued weight $T$ takes the following explicit form. Firstly, there exists a unique strongly continuous 1-parameter automorphism group $\mathbb{R} \rightarrow \Aut(\cN)$ called the {\it dual action}, determined by the formulas
\begin{equation}\label{Eqn=DualAction}
\theta_s(\pi(x)) = \pi(x), \quad \theta_s(\lambda(t)) = e^{-ist} \lambda(t), \qquad x \in \cM, s,t \in \mathbb{R}.
\end{equation}
%In fact,  define the unitary operator $w$ acting on $L^2(\mathbb{R}, \cH)$ by $(w \xi)(s) = e^{is} \xi(s), s \in \mathbb{R}$. Then, $\theta_s( y) = w^\ast y w$ for every $s \in \mathbb{R}$.
Further $\pi(\cM)$ is precisely the fixed point algebra $\{ y \in \cN \mid \forall s \in \mathbb{R}: \theta_s(y) = y \}$.  Thus $T$ is determined by the prescriptions
\begin{equation}\label{Eqn=TAlternative}
T: \cN^+ \rightarrow \cM^+_{{\rm ext}}: y \mapsto   \int_{\mathbb{R}} \theta_s(y) ds,
\end{equation}
and the dual weight of $\varphi$ is defined as $\tilde{\varphi} = \varphi \circ \pi^{-1} \circ T$. Note that we may identify $\cM$ with $\pi(\cM)$ via the normal map $\pi$; this gives the proper interpretation to \eqref{Eqn=TAlternative}.  One may then check that $\tilde{\varphi}$ is a normal, semi-finite weight on $\cN$ which is faithful if $\varphi$ is faithful. Clearly, $\tilde{\varphi} \circ \theta_s = \tilde{\varphi}$ for every $s \in \mathbb{R}$.

Finally note that in what follows we shall often need  a standard Fubini type argument, allowing us to exchange the $\sigma$-weak integrals of $\cM$-valued functions with a normal weight $\varphi$, contained in Lemma \ref{Lem=Fubini}. We shall sometimes use it  in the proofs without further reference.

\subsection{The approximating maps and the Haagerup property for crossed products}

In the remainder of this section we use the framework introduced above to study the Haagerup property and its behaviour with respect to passing to crossed products.

{

\begin{lem} \label{Lem=SemiFiniteCornerI}
Let $\cM, \cN, \varphi, \tilde{\varphi}, \alpha, G$ be as in Subsection \ref{Subsection:crossprel}. Assume moreover that $\alpha$ is $\varphi$-preserving. Let $f \in L^2(G)$ be such that $g \mapsto f \ast g, g \in L^2(G)$ extends to a bounded operator $\lambda(f)$ on $L^2(G)$. Moreover, suppose that   $\lambda(f) \geq 0$. Let $e_f \in \cN$ be the  support projection of $\lambda(f)$. The weight $\tilde{\varphi}_f(x) := \tilde{\varphi}(\lambda(f) x \lambda(f))$ restricted to $x \in e_f \cN e_f$ is semi-finite.
\end{lem}
\begin{proof}
Let $\{ x_i \}_{i \in I}$ be an increasing net of positive elements in $\nphi$ such that $x_i \nearrow 1$ strongly. Such a net exists by applying \cite[Proposition II.3.3]{TakII} to $\mphi$, which is contained in the left ideal $\nphi$. The map $\pi: \cM \rightarrow \cN$ is normal so that it preserves suprema and hence $\pi(x_i) \rightarrow 1$ strongly in $\cN$.  By Lemma \ref{Lem=HaagerupKrausComputation} we see that,
\[
\begin{split}
 & \tilde{\varphi}\left( \lambda(f)^\ast e_f   \pi(x_i)^\ast e_f \pi( x_i)  e_f \lambda(f) \right) \\
\leq &  \tilde{\varphi}\left( \lambda(f)^\ast    \pi(x_i)^\ast  \pi( x_i)    \lambda(f) \right)\\
=& \Vert   f \Vert_{L^2 G}^2 \:\:\varphi(x_i^\ast x_i) < \infty.
\end{split}
\]
This means that there exists a net in $e_f \: \mathfrak{n}_{\tilde{\varphi}_f} e_f$  converging strongly to $e_f$. Since,  $e_f \: \mathfrak{n}_{\tilde{\varphi}_f} e_f$ is a left ideal in  $e_f \cN e_f$ it is therefore strongly dense, which proves the lemma.
\end{proof}

}

\begin{prop}\label{Prop=CrossedProductEmbeddingUltra}
Let $\cM, \cN, \varphi, \tilde{\varphi}, \alpha, G$ be as in Subsection \ref{Subsection:crossprel}. Assume moreover that $\alpha$ is $\varphi$-preserving. Let $f \in L^2(G)$ with $\Vert f \Vert_{L^2(G)} = 1$ be such that $g \mapsto f \ast g, g \in L^2(G)$ extends to a bounded operator $\lambda(f)$ on $L^2(G)$. Assume that  $\lambda(f) \geq 0$. Let $e_f$ be the support projection of $\lambda(f)$ in $\cN$  and let $\tilde{\varphi}_f$ be the restriction of the weight
\begin{equation}\label{Eqn=WeightOnCorner}
y \mapsto \tilde{\varphi}(\lambda(f) y \lambda(f)), \qquad y \in \cN,
\end{equation}
to the corner algebra $e_f \cN e_f$ (we will slightly abuse the notation and denote by $\tilde{\varphi}_f$  also the weight on $\cN$ defined by the equality \eqref{Eqn=WeightOnCorner}, before the restriction). Then the following statements hold.
\begin{enumerate}
\item If $\varphi$ is faithful, then the weight $\tilde{\varphi}_f$ is faithful on $e_f \cN e_f$.
\item There exists a contractive map,
\[
U_f: L^2(\cM, \varphi) \rightarrow L^2(\cN, \tilde{\varphi}_f):\:\: \Lambda_{\varphi}(x) \mapsto   \Lambda_{\tilde{\varphi}_f} (e_f \pi(x) e_f).
\]
\end{enumerate}
\end{prop}
\begin{proof}
{ Begin by observing that the weight $\tilde{\varphi}_f$ is faithful, provided that $\varphi$ is faithful. To this end assume that $y\in e_f \cN e_f$ is such that
$\tilde{\varphi}_f (y^\ast y) =0$. Then, by faithfulness of $\tilde{\varphi}$, $y\lambda(f) =0$. But then, as $e_f$ is the projection onto the closure of the image of $\lambda(f)$, we also have $y e_f = 0$. Finally $y \in e_f \cN e_f$, so $y=0$.}

We shall now prove that $U_f$ is contractive. For $x \in \nphi$, since $\pi(x)^\ast e_f \pi(x) \leq \pi(x^\ast x)$, using Lemma \ref{Lem=HaagerupKrausComputation} and Lemma \ref{Lem=Fubini} yields
\begin{equation}
\begin{split}
\tilde{\varphi}_f(e_f \pi(x)^\ast e_f \pi(x) e_f) = & \tilde{\varphi}_f(\pi(x)^\ast e_f \pi(x) )\\
 \leq& \tilde{\varphi}_f( \pi(x^\ast x)) \\
= & \varphi\circ \pi^{-1} \circ T ( \lambda(f)^\ast\pi( x^\ast x) \lambda(f) )\\
 = & \varphi\left( \int_{G} \vert f(s) \vert^2  \alpha_s^{-1}(x^\ast x )ds \right) \\
=& \varphi(x^\ast x).
\end{split}
\end{equation}
This implies that $\Vert  \Lambda_{\tilde{\varphi}_f} (e_f \pi(x) e_f) \Vert_2 \leq \Vert \Lambda_\varphi(x) \Vert_2$ and hence $U_f$ is contractive.

\end{proof}

We now prove the following preliminary theorem. We shall improve on it later, once we have established that the Haagerup property is independent of the choice of the weight.

\begin{thm}\label{Thm=HPCrossedProduct}
Let $\cM$ be a von Neumann algebra with normal, semi-finite, faithful weight $\varphi$. Let $G$ be a locally compact group. Let $\alpha: G \rightarrow \Aut(\cM)$ be a strongly continuous  group of automorphisms such that $\varphi \circ \alpha =  \varphi$. Let $\cN := \cM \rtimes_{\alpha} G$ be the crossed product with dual weight $\tilde{\varphi} := \varphi \circ \pi^{-1} \circ T$.  %Assume for the moment that $G=\br$ and $\alpha$ is the action of the modular automorphism group of $\varphi$.
Then the following hold.
\begin{enumerate}
\item\label{Item=HPCrossedProduct0} Consider a net of functions $\{ f_j \}_{j\in J}$ in $L^2(G)$ such that $g \mapsto f_j \ast g, g \in C_c(G)$ extends to a bounded positive operator $\lambda(f_j)$ on $L^2(G)$. Assume moreover that $\Vert f_j \Vert_{L^2 G} = 1$ and that for every $g \in C_c(G)$ we have $\int_G \vert f_j(s)\vert^2 g(s) ds \rightarrow g(e)$. Assume also that for every $j\in J$ the pair $(e_j \cN e_j, \tilde{\varphi}_j)$ has the Haagerup property with cp maps  $\{ \Phi_k^{(j)} \}_{k \in \mathbb{N}}$ having $L^2$-implementations $T_k^{(j)}$ that are uniformly bounded both in $j$ and $k$. Then also $(\cM, \varphi)$ has the Haagerup property. Here $\tilde{\varphi}_j := \tilde{\varphi}_{f_j}$ is the weight constructed in Proposition \ref{Prop=CrossedProductEmbeddingUltra} and $e_j := e_{f_j}$ is also defined there.
\item\label{Item=HPCrossedProduct1} Consider the special case that $G = \mathbb{R}$ and $\alpha = \sigma^\varphi$. If $(\cN, \tilde{\varphi})$ has the Haagerup property %and $G$ is unimodular,
then $(\cM, \varphi)$ has the Haagerup property.
\item \label{Item=HPCrossedProduct2} Consider again the special case that $G = \mathbb{R}$ and $\alpha = \sigma^\varphi$. If $(\cM, \varphi)$  has the Haagerup property %and $G$ is unimodular,
then $(\cN, \tilde{\varphi})$ has the Haagerup property.
\end{enumerate}
\end{thm}
\begin{proof}
We start the proof by   showing how \eqref{Item=HPCrossedProduct1} and \eqref{Item=HPCrossedProduct2} relate back to the case \eqref{Item=HPCrossedProduct0}.

\vspace{0.3cm}

\eqref{Item=HPCrossedProduct1} We claim that the assumptions of \eqref{Item=HPCrossedProduct0} are automatically satisfied. Indeed one may take $f_j = \sqrt{\frac{j}{\pi}} e^{-\frac{1}{2} j t^2}$. Then $\lambda(f_j)$ is positive and bounded with support equal to 1  by the Fourier transform. It follows from Theorem \ref{Thm=HPSemiFinite} and its proof that $( \cN , \tilde{\varphi}_j)$ has the Haagerup property with uniform bounds on the $L^2$-implementations as in \eqref{Item=HPCrossedProduct0}. Thus from this point the proofs of  \eqref{Item=HPCrossedProduct0} and \eqref{Item=HPCrossedProduct1} proceed in exactly  the same way.

\vspace{0.3cm}

\eqref{Item=HPCrossedProduct2} We observe the following. Since $(\cM, \varphi)$ has the Haagerup property so has the pair given by $( \cM \wot \cB(L^2(\mathbb{R})), \varphi \otimes {\rm Tr})$, c.f.\  Lemma \ref{Lem=HPTensor} and Proposition \ref{Prop=HPBofH}. Since $\cM \wot \cB(L^2(\mathbb{R})) )$ is isomorphic to the crossed product $\cN \rtimes_\theta \mathbb{R}$ and in that case the double dual weight $\tilde{\tilde{\varphi}}$ corresponds to $\varphi \otimes {\rm Tr}$ (see \cite[Theorem X.2.3]{TakII}) it suffices to find functions $f_j$ such that \eqref{Item=HPCrossedProduct0} is satisfied for $\alpha = \theta$ and $G = \mathbb{R}$. We let $\hat{f}_j = \vert F_j \vert^{-\frac{1}{2}}  \chi_{F_j}$ where $F_j = [-j, j]$ (they are F\o{}lner sets) and let $f_j$ be its Fourier transform and note that $\Vert f_j \Vert_{L^2(\mathbb{R})} = 1$ by Plancherel's identity. Then $\lambda(f_j) = \nu( \hat{f}_j )$ (by \cite[p. 259 (13)]{TakII} and the Fourier transform). Let $e_j$ be the support projection of $\lambda(f_j)$ and note that in fact $e_j = \nu(\chi_{F_j})$, again by \cite[p.259 (13)]{TakII}. Let $( \cM \wot \cB(L^2(\mathbb{R})),  \varphi \otimes {\rm Tr})$ have the Haagerup property with approximating cp maps $\Phi_k$ and $L^2$-implementations $T_k$. Then $( e_j(\cM \wot \cB(L^2(\mathbb{R}))e_j,  \tilde{\varphi}_j)$ has the Haagerup property with approximating cp maps
\[
\nu(f_j)^{-1} \Phi_k(\nu(f_j) \: \cdot \:   \nu(f_j))\nu(f_j)^{-1}.
 \]
 Since $\nu(f_j)$ is a multiple of the the identity in the corner algebra $e_j(\cM \wot \cB(L^2(\mathbb{R}))e_j$ the latter map is just $e_j \Phi_k(e_j \: \cdot \: e_j)e_j$. The $L^2$-implementations of the latter map are bounded by $\sup_k \Vert T_k \Vert$ as $e_j$ is in the centralizer of $\tilde{\varphi}$. By the fact that $f_j(s) = \sqrt{j} f_1(js)$ it follows that $\int_G \vert f_j(s)\vert^2 g(s) ds \rightarrow g(e)$ for every $g \in C_c(G)$. Again the proof now proceeds exactly as in case \eqref{Item=HPCrossedProduct0}.
%}

\vspace{0.3cm}

So from this point we shall assume \eqref{Item=HPCrossedProduct0}.

\noindent {\bf Step 0: Construction of completely positive maps.} Let for each $j \in J$ the sequence $\{ \Phi_k^{(j)} \}_{k \in \mathbb{N}}$ be the sequence of completely positive maps witnessing the Haagerup property of $(e_j \cN e_j, \tilde{\varphi}_j)$. That is, $\tilde{\varphi}_j\circ \Phi_k^{(j)} \leq \tilde{\varphi}_j$, the $L^2$-identification of the map $\Phi_k^{(j)}$, denoted by $T_k^{(j)}$, is compact and  $T_k^{(j)} \rightarrow 1$ strongly as $k \rightarrow \infty$. For each $k \in \bn$, $j \in J$ define a normal, completely positive map $\cM \rightarrow \cM$ by the formula
\[
\Psi_k^{(j)}(x) =  \pi^{-1} \circ  T\left( \lambda(f_j) \Phi_k^{(j)}(  e_j \pi(x) e_j ) \lambda(f_j) \right), \qquad x \in \cM.
\]

\noindent {\bf Step 1: Verifying the first criterium of Definition \ref{Dfn=HP}}.  We have the following estimate, in which we use respectively the definition of $\Psi_k^{(j)}$, the defining property of $\Phi_k^{(j)}$, Lemma \ref{Lem=HaagerupKrausComputation} and the assumption that $\varphi \circ \alpha =   \varphi$: for each $x \in \cM^+$
\begin{equation}\label{Eqn=StatePreserving}
\begin{split}
\varphi ( \Psi_k^{(j)}(x)  ) = &  \varphi\circ \pi^{-1} \left( T (  \lambda(f_j) \Phi_k^{(j)}(e_j \pi(x) e_j) \lambda(f_j)  ) \right)
\leq  \varphi \circ \pi^{-1}( T( \lambda(f_j)  e_j\pi( x) e_j  \lambda(f_j) ) )\\
 = &  \varphi\circ \pi^{-1}( T( \lambda(f_j)   \pi( x)    \lambda(f_j) ) )
=   \varphi \left( \int_{G} \vert f_j(s)\vert^2 \alpha^{-1}_s(x) ds \right)
  = \varphi(x).
\end{split}
\end{equation}
This shows that $\Psi_k^{(j)}$ satisfies Definition \ref{Dfn=HP} \ref{Item=HPOne} for $(\cM, \varphi)$, for every $j \in J$ and $k \in \mathbb{N}$.

\vspace{0.3cm}

\noindent {\bf Step 2: Verifying the second criterium of Definition \ref{Dfn=HP}}. We shall prove that each $\Psi_k^{(j)}$ also satisfies Property \ref{Item=HPTwo} of Definition \ref{Dfn=HP}. That is, we shall show that for each $j \in J$, $k \in \mathbb{N}$, the formula
\[
S_k^{(j)}: \Lambda_{\varphi}(x) \mapsto \Lambda_{\varphi}(\Psi_k^{(j)}(x)), \qquad \forall \: x \in \nphi,
\]
determines a compact operator on $L^2(\cM, \varphi)$, and the corresponding operator norms are uniformly bounded.
% Moreover, there exist $k,l \in \mathbb{N}$ such that,
%\[
%\Vert \Lambda_{\varphi}(x) - \Lambda_{\varphi}(\Psi_k^{(l)}(x)) \Vert_2 \leq \frac{1}{\vert F \vert}, \qquad \forall \: x \in F.
%\]
%Here, $\vert F \vert$ denotes the number of elements in the set $F$. This then concludes the proof of the theorem.
In order to do so, we need two auxiliary maps:
\begin{itemize}
\item Firstly, by Proposition \ref{Prop=CrossedProductEmbeddingUltra} there exists a contraction,
\[
U_j: L^2(\cM, \varphi) \rightarrow L^2(e_j \cN e_j, \tilde{\varphi}_j): \Lambda_{\varphi}(x) \mapsto \Lambda_{\tilde{\varphi}_j}(e_j \pi( x) e_j).
\]
\item Secondly, by Lemma \ref{Lem=KadisonSchwarz} and Lemma \ref{Lem=HaagerupKrausComputation} we may introduce a contractive map, using that $\Vert f_j \Vert_{L^2 G} = 1$,
\[
A_j: L^2(e_j \cN e_j, \tilde{\varphi}_j) \rightarrow L^2(\cM, \varphi): \Lambda_{\tilde{\varphi}_j}(x) \mapsto \Lambda_{\varphi}( \pi^{-1} \circ T(\lambda(f_j) x \lambda(f_j) ) ).
\]

\end{itemize}

%Firstly, note that Lemma \ref{Lem=HaagerupKrausComputation} shows that for $x \in \nphi$, we have,
%\[
%\begin{split}
%\Vert \Lambda_{\tilde{\varphi}_i}(x) \Vert_2^2 = & \tilde{\varphi}(\lambda(f_i)^\ast x^\ast x \lambda(f_i) ) = \varphi( T( \lambda(f_i)^\ast x^\ast x \lambda(f_i) )) \\
%= & \varphi( \int_{\mathbb{R}} \vert f_i(s) \vert^2  \alpha_{-s}(x^\ast x) ds ) = \varphi(x^\ast x) = \Vert \Lambda_{\varphi}(x) \Vert_2^2.
%\end{split}
%\]
%Therefore, the GNS-map of $\varphi$ may be related to the GNS-map of $\tilde{\varphi}$ as follows. For every $i$, there exists an isometric map:

Now, we claim that we have the following intertwining property:
\begin{equation}\label{Eqn=IntertwiningProperty}
S_{k}^{(j)} =  A_j T_k^{(j)} U_j, \;\;\; j \in J, k \in \mathbb{N}.
\end{equation}
Since $T_k^{(j)}$ is compact, this then implies that $S_k^{(j)}$ is compact. To prove the claim, let $x \in \nphi$. Then,
\[
\begin{split}
& A_j T_k^{(j)} U_j \Lambda_{\varphi}(x) \\
=&  A_j T_k^{(j)} \Lambda_{\tilde{\varphi}_j}(e_j \pi(x) e_j) \\
= &  A_j \Lambda_{\tilde{\varphi}_j}( \Phi_k^{(j)}(e_j \pi(x) e_j)) \\
=& \Lambda_{\varphi} (  \pi^{-1} \circ T(\lambda(f_j)\Phi_k^{(j)}(e_j \pi(x) e_j)  \lambda(f_j) )   ) \\
 = & \Lambda_{\varphi}( \Psi_k^{(j)}(x) ) \\
= & S_k^{(j)} \Lambda_{\varphi}(x).
\end{split}
\]

\vspace{0.3cm}

\noindent {\bf Step 3: Verifying the third criterium of Definition \ref{Dfn=HP}}. We now arrive at choosing $j \in J$ and $k   \in \mathbb{N}$ so that we construct a suitable approximating sequence of maps out of the maps $\Psi_k^{(j)}$. By separability of $\cM$, one may find an increasing sequence of finite subsets of $\nphi$ such that the union of the corresponding images in $L^2(\cM, \varphi)$ is dense.  The proof proceeds by choosing appropriate $j \in J$ and $k \in \mathbb{N}$ so that the resulting sequence $\Psi_k^{(j)}$ (indexed by finite subsets mentioned above) witnesses the Haagerup property for $(\cM, \varphi)$. For the rest of the proof we fix  a finite subset $F \subseteq \nphi$.

Firstly, we claim that there exists a $j \in J$ such that for every $x \in F$ we have,
\[
\Vert A_j \Lambda_{\tilde{\varphi}_j}(e_j \pi(x) e_j) - \Lambda_{\varphi}(x) \Vert_2 \leq \frac{1}{2 \vert F \vert}.
\]
Here $\vert F \vert$ is the number of elements in $F$. Indeed, using Lemma \ref{Lem=HaagerupKrausComputation},
\[
\begin{split}
& \Vert A_j \Lambda_{\tilde{\varphi}_j}( e_j \pi(x)  e_j) -    \Lambda_{\varphi}(x) \Vert_2 \\
= &\Vert  \Lambda_{\varphi}(\pi^{-1}\circ T(\lambda(f_j) \pi(x)  \lambda(f_j) ) ) - \Lambda_{\varphi}(x) \Vert_ 2 \\
= & \Vert   \int_{\mathbb{R}} \vert f_j(s) \vert^2  \Lambda_{\varphi}( \alpha_{s}^{-1}(x)    )ds  - \Lambda_{\varphi} (x) \Vert_ 2
\end{split}
\]
and the corresponding limit in $j$ tends to 0 by Lemma \ref{Lem=StronglyCTGroup}. Next, we may choose $k \in \mathbb{N}$ such that for every $x \in F$ we have
\[
\Vert T_k^{(j)} \Lambda_{\tilde{\varphi}_j}( e_j \pi(x)  e_j) - \Lambda_{\tilde{\varphi}_j}(e_j \pi(x)  e_j) \Vert_2 \leq \frac{1}{2 \vert F \vert},
\]
by definition of $T_k^{(j)}$. In that case, for every $x \in F$, we have,
\[
\begin{split}
& \Vert A_j T_k^{(j)} \Lambda_{\tilde{\varphi}_j}(e_j \pi(x) e_j) - \Lambda_{\varphi}(x) \Vert_2\\
 \leq &
\Vert A_j T_k^{(j)} \Lambda_{\tilde{\varphi}_j}(e_j \pi(x)  e_j) - A_j \Lambda_{\tilde{\varphi}_j}(e_j \pi(x)  e_j) \Vert_2 + \Vert A_j \Lambda_{\tilde{\varphi}_j}(e_j \pi(x)  e_j) -  \Lambda_{\varphi}( x) \Vert_2 \leq \frac{1}{\vert F \vert}.
\end{split}
\]
And hence, also for all $x \in F$,
\[
\Vert  A_j T_k^{(j)} U_j \Lambda_{\varphi}(x) - \Lambda_\varphi(x) \Vert \leq \frac{1}{\vert F\vert}.
\]
This ends the proof.
\end{proof}

\subsection{Main results of the section}

\begin{thm} \label{Thm=InvariantOfTheVNA}
Let $\cM$ be a  von Neumann algebra (recall with a separable predual) and let $\varphi$ and $\psi$ be two normal, semi-finite, faithful weights on $\cM$. Then, $(\cM, \varphi)$ has the Haagerup property if and only if $(\cM, \psi)$ has the Haagerup property. That is, the Haagerup property is an invariant of $\cM$.
\end{thm}
\begin{proof}
Assume that $(\cM, \varphi)$ has the Haagerup property and let $\sigma^{\varphi}$ be the modular automorphism group of $\varphi$ and $\theta$ the dual action defined in \eqref{Eqn=DualAction}.  By crossed product duality \cite[Theorem X.2.3]{TakII}, there is a canonical isomorphism,
\[
 (\cM \rtimes_{\sigma^{\varphi}} \mathbb{R}) \rtimes_\theta \mathbb{R} \simeq \cM \wot \cB(L^2\mathbb{R}),
\]
such that  the second dual weight $\tilde{\tilde{\varphi}}$ corresponds to $\varphi \otimes {\rm Tr}$. Since $(\cM, \varphi)$ has the Haagerup property, it follows that $( \cM \wot \cB(L^2(\mathbb{R})), \varphi \otimes {\rm Tr})$ has the Haagerup property, c.f.\  Lemma \ref{Lem=HPTensor} and Proposition \ref{Prop=HPBofH}. Furthermore, it is clear from the definition that the dual action $\theta$ is a strongly continuous $\tilde{\varphi}$-preserving 1-parameter group of automorphisms of $\cM \rtimes_{\sigma^{\varphi}} \mathbb{R}$. Therefore, Theorem \ref{Thm=HPCrossedProduct} implies that $(\cM \rtimes_{\sigma^{\varphi}} \mathbb{R}, \tilde{\varphi})$ has the Haagerup property. By \cite[Proposition 3.5]{Tak73}, we have the following isomorphism of von Neumann algebras:
\[
\cM \rtimes_{\sigma^\varphi} \mathbb{R} \simeq \cM \rtimes_{\sigma^\psi} \mathbb{R}.
\]
The von Neumann algebra $\cM \rtimes_{\sigma^\varphi} \mathbb{R}$ is well-known to be semi-finite (this follows, as the modular automorphism group $t\mapsto \sigma_t^{\tilde{\varphi}}$ coincides with the 1-parameter group $t \mapsto {\rm Ad} \:\lambda(t)$ consisting of inner automorphisms, see the proof of \cite[Theorem X.1.1]{TakII}). Theorem \ref{Thm=HPSemiFinite} implies that $(\cM \rtimes_{\sigma^\psi} \mathbb{R}, \tilde{\psi})$ has the Haagerup property. Then, since $\sigma^\psi$ is a strongly continuous $\psi$-preserving 1-parameter group of automorphisms we may apply Theorem \ref{Thm=HPCrossedProduct} to see that $(\cM, \psi)$ has the Haagerup property.
 \end{proof}

\begin{rmk}
From this point we will simply say that a von Neumann algebra $\cM$ has {\it the} Haagerup property, without specifying the weight (or state) with respect to which this property holds. The independence of the weight allows us to improve on several preliminary results.
\end{rmk}

\begin{prop}\label{Lem=DescentDownToTensor}
Let $\cM$ and $\cN$ be (recall $\sigma$-finite) von Neumann algebras. If $\cM \wot \cN$ has the Haagerup property, then so do $\cM$ and $\cN$.
\end{prop}
\begin{proof}
Fix a faithful, normal state $\varphi$ on $\cM$ and a faithful, normal state $\psi$ on $\cN$.  Let $(\Phi_k)_{k \in \bn}$ be the sequence of maps that witnesses the Haagerup property for $(\cM \wot \cN, \varphi \otimes \psi)$. Let $T_k$ denote their $L^2$-implementations.  Then $\Psi_k := (\varphi \otimes \id )\circ \Phi_k$ defines the sequence of maps witnessing the Haagerup property for $\cN \simeq 1 \otimes \cN \subseteq \cM \wot \cN$ equipped with the state $\psi$. Indeed, trivially $\psi \circ \Psi_k \leq \psi$ so that condition \ref{Item=HPOne} of Definition \ref{Dfn=HP}  is satisfied. Consider the isometry
\[
U: L^2(\cN, \psi) \rightarrow L^2(\cM \wot \cN, \varphi \otimes \psi): \xi \mapsto \Lambda_{\varphi}(1_{\cM}) \otimes \xi.
\]
The $L^2$-implementation of $\Psi_k$ is then  given by $U^\ast T_k U$. Indeed, for $x \in \cN$,
\[
\begin{split}
U^\ast T_k U \Lambda_{\psi}(x) = & U^\ast T_k \Lambda_\varphi(1_{\cM}) \otimes \Lambda_\psi(x) = U^\ast (\Lambda_\varphi \otimes \Lambda_\psi)(\Phi_k(1_{\cM} \otimes x))\\
 = & \Lambda_{\varphi}( (\varphi \otimes \id)(\Phi_k(1_{\cM} \otimes x) ) ) = \Lambda_{\psi}(\Psi_k(x)).
\end{split}
\]
This yields properties \ref{Item=HPTwo} and \ref{Item=HPThree} of Definition \ref{Dfn=HP}.
\end{proof}

{
\begin{prop}\label{Prop=HPCornerNonModular}
Let $\cN$ be a von Neumann algebra and let $e \in \cN$ be a projection.  If $\cN$ has the Haagerup property, then so does $e \cN e$.
\end{prop}
\begin{proof}
Let $\varphi$ be a normal, semi-finite, faithful weight on $\cN$ with support equal to $e$. Let $\psi$ be a normal, semi-finite weight on $\cN$ with support equal to $1-e$. Set $\rho = \varphi + \psi$. Then $\rho$ is a normal, semi-finite, faithful weight on $\cN$ and $e$ is contained in the centralizer of $\rho$. The lemma follows then from Lemma \ref{Lem=HPCorner}.
\end{proof}
}

%{\color{blue}
It is important to note that it follows directly from the constructions in our proofs that not only $\cM$ has the Haagerup property if and only if its core $\cM \rtimes_\sigma \mathbb{R}$ does so; moreover the bounds on the cp approximating maps ($\Phi_k$ of Definition \ref{Dfn=HP}) do not increase. Since for a semi-finite von Neumann algebra these approximating maps may in fact be chosen contractive, this holds in fact for any von Neumann algebra (see \cite{CasSkaII} for a similar argument). In particular, the assumption \eqref{Item=HPCrossedProduct0} on the bounds of $T_k^{(j)}$ is redundant. This allows us to conclude the following.

\begin{thm}\label{Prop=HPAllGroups}
Consider $G, \alpha, \cM, \cN, \varphi, \tilde{\varphi}$  arbitrary as in the statement of Theorem \ref{Thm=HPCrossedProduct}. If $(\cN, \tilde{\varphi})$ has the Haagerup property then so does $(\cM, \varphi)$.
\end{thm}
\begin{proof}
We use the notation of Theorem \ref{Thm=HPCrossedProduct}. Suppose that $\cN$ has the Haagerup property. It follows then from Proposition \ref{Prop=HPCornerNonModular} that for {\it any} $f_l \in C_c(G)$ such that $\Vert f_l\Vert_{L^2 G}=1$ and $\lambda(f_l)\geq 0$, the von Neumann algebra $e_l \cN e_l$, where $e_l$ is the support projection of $\lambda(f_l)$, has the Haagerup property. So if $\cN$ has the Haagerup property, then the conditions of Theorem \ref{Thm=HPCrossedProduct} \eqref{Item=HPCrossedProduct0} are satisfied (the assumption on the bounds of $T_k^{(j)}$ follow from the considerations preceding this theorem) and hence $\cM$ has the Haagerup property.
\end{proof}

\section{From the von Neumann algebra to the crossed product}\label{Sect=ToCrossedProduct}

In this section we present the converse of our results in Section \ref{Sect=TypeIII} for amenable groups: we will prove that if an amenable group $G$ acts on a von Neumann algebra $\cM$ that has the Haagerup property, then the crossed product $\cN = \cM \rtimes G$ has the Haagerup property. Our proofs rely on similar techniques as those in Section \ref{Sect=TypeIII}. However, since the situation here is (Pontrjagin) dual to Section \ref{Sect=TypeIII}, some of our arguments change.

We use again the notation of the Subsection \ref{Subsection:crossprel}.

\vspace{0.3cm}

We start with the following lemma. Part of the statement was already proved in \cite[Lemma 3.9]{HaaKrau}.

\begin{lem}\label{Lem=HaagerupKrausDual}
Let $x \in C_c(G, \cM)$ and let $f \in C_c(G)$. Define, for each $s,t \in G$, the element $a_t(s) = f(s^{-1}t) x(s)\in \cM$ and set
\[
S_f(y) = S( \nu(\overline{f}) y \nu(f) ), \qquad y \in \cM \wot \cB(L^2 G).
\]
Then, $S_f$ is a normal completely positive map from $\cM \wot\cB(L^2G)$ to $\cN$. Moreover, if $\chi_F$ denotes the characteristic function of a measurable set $F \subseteq G$, then
\begin{equation}\label{Eqn=SfExpression}
S_f( \mu(x^\sharp) \nu( \chi_F ) \mu(x) ) = \int_F \mu(a_t^\sharp \ast a_t ) dt.
\end{equation}
This implies in particular that if $\varphi$ is an $\alpha$-preserving normal, semi-finite, faithful weight on $\cM$ with dual weight $\tilde{\varphi}$ on $\cN$,  then,
\begin{equation}\label{Eqn=VarphiSfExpression}
\tilde{\varphi}\circ S_f( \mu(x^\sharp) \nu( \chi_F ) \mu(x) ) =    \varphi\left( \int_G \int_F \vert f(s^{-1}t)  \vert^2 dt\:  x(s)^\ast x(s)     ds  \right)
\end{equation}
Moreover for every $y \in \mathfrak{n}_{\tilde{\varphi}}$ we have
\begin{equation}\label{Eqn=PhiSWeightIneq}
\tilde{\varphi}\circ S_f( y^\ast y ) = \Vert f \Vert_{L^2G}^2 \:\tilde{\varphi}( y^\ast y ),
\end{equation}
so that in particular $\tilde{\varphi} \circ S_f|_{\cN} \leq \Vert f \Vert_{L^2 G}^2 \tilde{\varphi}$.
\end{lem}
\begin{proof}
We may assume that $f \neq 0$, as otherwise the claims of the lemma are trivially satisfied. The fact that $S_f$ defines a normal, completely positive map was proved in \cite[Lemma 3.9]{HaaKrau}. Thus we only need to prove \eqref{Eqn=SfExpression} and \eqref{Eqn=VarphiSfExpression} (the last two statements in the lemma follow then easily by putting $F=G$). Firstly, we have the following series of equations (note we use the fact that images of the maps $\pi$ and $\nu$ commute):
\begin{equation}\label{Eqn=InitialSfComp}
\begin{split}
& S_f( \mu( x^\sharp) \nu(\chi_F) \mu( x))\\
= & \int_G \beta_t \left( \nu(\overline{f}) \mu(x^\sharp) \nu(\chi_F) \mu( x) \nu(f) \right) dt \\
= & \int_G \nu(\overline{f_t}) \mu(x^\sharp)\nu(\chi_{Ft^{-1}}) \mu( x) \nu(f_t) dt \\
= & \int_G \nu(\overline{f_t} ) \left( \int_G \lambda(s) \pi(x(s)) ds \right)^\ast \nu(\chi_{Ft^{-1}}) \int_G \lambda(s) \pi( x(s) ) ds \:\: \nu(f_t) dt \\
= &  \int_G\int_G\int_G \nu(\overline{f_t} )  \pi(x(r))^\ast \lambda(r^{-1}) \nu(\chi_{Ft^{-1}})  \lambda(s) \pi( x(s) )   \:\: \nu(f_t) dt \: ds \:dr \\
= &  \int_G\int_G\int_G  \pi(x(r))^\ast \lambda(r^{-1}) \nu(\overline{_{r^{-1}}f_t} ) \nu(\chi_{Ft^{-1}}) \nu( _{s^{-1}}f_t  )    \lambda(s) \pi( x(s) )   \:\:  dt \: ds \:dr \\
\end{split}
\end{equation}
Next, we see that for $g \in G$,
\[
\int_G   \overline{_{r^{-1}}f_t}(g) \chi_{Ft^{-1}}(g) \: _{s^{-1}}f_t(g) dt = \int_G \overline{f(r^{-1}g t)}  \chi_F(gt) f(s^{-1} g t) dt = \int_F   \overline{f_t(r^{-1})}   f_t(s^{-1}) dt.
\]
This allows us to continue the computation \eqref{Eqn=InitialSfComp}:
\begin{equation}\label{Eqn=JustSomeComputation}
\begin{split}
& S_f( \mu( x^\sharp) \nu(\chi_F) \mu( x)) \\
= &  \int_G\int_G\int_F  \pi(x(r))^\ast \lambda(r^{-1}) \overline{f_t(r^{-1}) }   f_t(s^{-1})       \lambda(s) \pi( x(s) )   \:\:  dt \: ds \:dr \\
= & \int_F \mu(a_t^\sharp) \mu(a_t) dt
=  \int_F \mu(a_t^\sharp \ast a_t) dt.
\end{split}
\end{equation}
This proves \eqref{Eqn=SfExpression}. Next we choose a measurable $F\subseteq G$ and compute
\[
\begin{split}
& \tilde{\varphi} \circ S_f ( \mu( x^\sharp) \nu(\chi_F) \mu( x)) \\
= & \tilde{\varphi}\left( \int_F \mu( a_t^\sharp \ast a_t  ) dt \right) \\
= & \int_F \tilde{\varphi}( \mu( a_t^\sharp \ast a_t  ) ) dt \\
= & \int_F \varphi(   (a_t^\sharp \ast a_t)(e)     ) dt \\
= & \int_F \varphi\left( \int_G \vert f_t(s^{-1})\vert^2   x(s)^\ast x(s)      ds  \right) dt \\
= &  \varphi\left( \int_G \int_F \vert f(s^{-1}t)  \vert^2 dt \:  x(s)^\ast x(s)    ds  \right).
\end{split}
\]
As suggested above, if $F = G$  the above expression reduces to
\[
\begin{split}
\tilde{\varphi} \circ S_f ( \mu( x^\sharp) \mu( x)) = &  \Vert f \Vert_{L^2G}^2   \:  \varphi\left( \int_G x(s)^\ast x(s)     ds  \right)  \\
= & \Vert f \Vert_{L^2G}^2 \:\tilde{\varphi}(\mu(x^\sharp \ast x) ).
\end{split}
\]
Thus we obtain an isometric map $\Lambda_{\tilde{\varphi}}(\mu(x)) \mapsto \Vert f \Vert_{L^2 G}^{-1} \Lambda_{\tilde{\varphi} \circ S_f}(\mu(x))$ defined for $x \in B_{\varphi}$. By Lemma \ref{Lem=CoreLemma} this extends to an isometry from $L^2(\cN, \tilde{\varphi}) \rightarrow L^2(\cN, \tilde{\varphi} \circ S_f)$.  Lemma \ref{Lem=CoreLemma} implies also that for each $y \in \mathfrak{n}_{\tilde{\varphi}}$ we may find a net $\{ x_i \}$ in $B_\varphi$ such that $\mu(x_i)$ converges $\sigma$-weakly to $y$ and  $\Lambda_{\tilde{\varphi}}(\mu(x_i))$ converges in norm to $\Lambda_{\tilde{\varphi}}(y)$. This implies that $\Lambda_{\tilde{\varphi}\circ S_f}(\mu(x_i))$ is a Cauchy net and since  $\Lambda_{\tilde{\varphi}\circ S_f}$ is $\sigma$-weak/norm closed, this means that also $y \in \mathfrak{n}_{\tilde{\varphi} \circ S_f}$ and $\tilde{\varphi}\circ S_f(y^\ast y) = \Vert f \Vert_{L^2 G}^2 \tilde{\varphi}(y^\ast y)$.

\begin{comment}
This expression reduces to \eqref{Eqn=PhiSWeightIneq} in case of $F = G$. Now, take an arbitrary $y \in \mathfrak{n}_\varphi$, for which \eqref{Eqn=PhiSWeightIneq} requires a slightly different argument. Let $\Lambda: L^2G \cap L^\infty G \rightarrow L^2 G$ be the GNS-map of the Haar measure on $G$. In that case,
\[
\begin{split}
\tilde{\varphi} \circ S_f(y^\ast y ) = & \tilde{\varphi} \circ S(\nu(\overline{f}) y^\ast y \nu(f) )
 =  \Vert \Lambda_{\tilde{\varphi}}(y) \otimes \Lambda(f) \Vert_2^2
=  \Vert  \Lambda_{\tilde{\varphi}}(y) \Vert_2^2 \Vert f \Vert_{L^2 G}^2
= \tilde{\varphi}(y^\ast y)  \Vert f \Vert_{L^2 G}^2.
\end{split}
\]
\end{comment}

\end{proof}

Note that we are not claiming that the weights $\tilde{\varphi} \circ S_f$ and $\Vert f \Vert^2_{L^2 G} \tilde{\varphi}$  coincide, but only that their values coincide on $\mathfrak{m}_{\tilde{\varphi}}$, i.e. the linear span of $\mathfrak{n}_{\tilde{\varphi}}^\ast \mathfrak{n}_{\tilde{\varphi}}$.

\begin{dfn}
An approximating sequence in $C_c(G)$ is a sequence $(g_n)_{n=1}^{\infty}$ of non-negative elements in $C_c(G)$ such that for each $n \in \bn$ we have $\|g_n\|_1=1$ and
for each $s \in G$ we have
\[ \lim_{n\to \infty} \|\delta_s \ast g_n - g_n\|_1 = 0.\]
\end{dfn}

It is well-known (see for example \cite[Proposition 0.8]{Pat}) that $G$ is amenable if and only if $G$ admits an approximating sequence in $C_c(G)$.

{

The proof of the following Lemma \ref{Lem=SemiFiniteCornerII} is almost exactly the same as Lemma \ref{Lem=SemiFiniteCornerI}. The difference is mainly that it involves an extra degree of crossed product duality.

\begin{lem} \label{Lem=SemiFiniteCornerII}
Let $f \in C_c(G)$ with support $F \subseteq G$. Let $e = \nu(\chi_F)$ be the support projection of $\nu(f)$. The weight $\tilde{\varphi} \circ S_f$ on $e (\cM \otimes \cB(L^2 G)) e$ is semi-finite.
\end{lem}
\begin{proof}
Let $\{ x_i \}_{i \in I}$ be an increasing net of positive elements in $\nphi$ such that $x_i \nearrow 1$ strongly. Such a net exists by applying \cite[Proposition II.3.3]{TakII} to $\mphi$, which is contained in the left ideal $\nphi$. The map $\pi: \cM \rightarrow \cN$ is normal so that it preserves suprema and hence $\pi(x_i) \rightarrow 1$ strongly in $\cN$. Let $\{h_j\}_{j \in J}$ be a net in $C_c(G)$ such that $\Vert h_j \Vert_1 = 1$ and with supports shrinking to the unit element of $G$. Then, $\lambda(h_j) \rightarrow 1$ strongly in $\cN$. Define $a_\alpha(s) = h_j(s) x_i$ with $\alpha = (i,j) \in I \times J$ with the natural net structure. By definition $\mu(a_\alpha) = \lambda(h_j) \pi(x_i)$. Since the product is strongly continuous on bounded sets, we find that $\mu(a_\alpha) \rightarrow 1$ strongly in $\cN$, hence in $\cM \wot \cB(L^2 G)$ which acts on the same Hilbert space. By Lemma \ref{Lem=HaagerupKrausDual} we see that
\[
\begin{split}
\tilde{\varphi} \circ S_f \left( e \mu(a_\alpha)^\ast e \mu(a_\alpha) e \right) = & \varphi\left( \int_G \int_F \vert h_j(s^{-1}t) \vert^2 dt\: x_i^\ast x_i\: ds \right) \\
= &\int_G \int_F \vert h_j(s^{-1} t) \vert^2 dt ds \: \varphi(x_i^\ast x_i) < \infty.
\end{split}
\]
This means that $e \mathfrak{n}_{\tilde{\varphi} \circ S_f} e$ contains a net that converges strongly to $e$. Since  $e \mathfrak{n}_{\tilde{\varphi} \circ S_f} e$ is a left ideal in  $e (\cM \wot\cB(L^2 G)) e$, it is therefore strongly dense in the latter algebra, which proves the lemma.
\end{proof}
}

\begin{prop}\label{Prop=CrossedProductEmbeddingUltraDual}
Let $\cM, \cN, \varphi, \tilde{\varphi}, \alpha, G$ be as in Section \ref{Sect=TypeIII}.  Let $f \in C_c(G)^+$  with $\Vert f \Vert_{L^2(G)} = 1$ and support equal to $F$ and set $e_f = \nu(\chi_{F})$. Let $\varphi_f$ be the restriction of the weight $\tilde{\varphi} \circ S_{f}$
to the corner algebra $e_f (\cM \wot B(L^2 G)) e_f$.   There exists a contractive map,
\[
V_f: L^2(\cN, \tilde{\varphi}) \rightarrow L^2(e_f(\cM \otimes \mathcal{B}(L^2 G)) e_f, \varphi_f):\:\: \Lambda_{\tilde{\varphi}}(x) \mapsto   \Lambda_{\varphi_f} (e_f x e_f).
\]
\end{prop}
\begin{proof}
For $y \in \mathfrak{n}_{\tilde{\varphi}}$ we find by \eqref{Eqn=PhiSWeightIneq} that
\[
\Vert \Lambda_{\varphi_f}(e_f y e_f) \Vert_2^2 =  \varphi_f(y^\ast e_f y) \leq \varphi_f(y^\ast y) = \tilde{\varphi}(y^\ast y) = \Vert \Lambda_{\tilde{\varphi}}( y ) \Vert_2^2.
\]
This proves that each $V_f$ is contractive.
\end{proof}

\begin{lem} \label{Lem=SFLStrongToOne}
Suppose that $G$ is amenable and let $(g_n)_{n=1}^{\infty}$ be an approximating sequence in $C_c(G)$. Define  for each $n \in \bn$ function $f_n = g_n^{\frac{1}{2}}$.  Then, for every $y \in \mathfrak{n}_{\tilde{\varphi}}$,
\[
\lim_{n \rightarrow \infty} \Vert \Lambda_{\tilde{\varphi}}(S_{f_n}(y)) - \Lambda_{\tilde{\varphi}}(y) \Vert_2 = 0.
\]
\end{lem}
\begin{proof}
Note that $S_{f_n}(y) \in \mathfrak{n}_{\tilde{\varphi}}$ by Lemma \ref{Lem=HaagerupKrausDual} and Lemma \ref{Lem=KadisonSchwarz}. In particular Lemma \ref{Lem=HaagerupKrausDual} implies that the family of maps $\Lambda_{\tilde{\varphi}}(y) \mapsto  \Lambda_{\tilde{\varphi}}(S_{f_n}(y))$ is jointly bounded. We need to prove that it converges to 1 strongly.

Let $x \in B_\varphi$  and put $y = \mu(x^\sharp \ast x)$. By Lemma \ref{Lem=CoreLemma} it suffices then to prove the equality displayed in the lemma for such $y$.  Put also $a_{n,t}(s) = f_n(s^{-1} t) x(s),\, t,s \in G, n \in \mathbb{N},$ and define the $\cM$-valued functions $b, b_n$ on $G$ by the formulas ($n \in \bn$, $u \in G$)
\[
\begin{split}
b_n(u) & = \int_G f_n(ut) f_n(t)  dt (x^\sharp \ast x)(u),\\
b(u) & = (x^\sharp \ast x)(u).
\end{split}
\]
%It should be noted that the integral appearing in the definition of $b_n$ exists by the Cauchy-Schwartz inequality applied to $f_n$ and $f_n$ shifted by $u$.
% Clearly, $\mu(x^\sharp \ast x) = \mu(b)$.
Using \eqref{Eqn=JustSomeComputation} and the definition of $\mu$ we see that
\[
\int_G (a_{n,t}^\sharp \ast a_{n,t}) dt = \int_G \int_G \pi(x(r))^\ast \lambda(r^{-1})  \int_G f_n(r^{-1}t) f_n(s^{-1}t)  dt \lambda(s) \pi(x(s)) ds dr = \mu(b_n).
\]
Hence by \eqref{Eqn=SfExpression}
\[
\begin{split}
& \Vert \Lambda_{\tilde{\varphi}}(S_{f_n}(y)  ) - \Lambda_{\tilde{\varphi}}( y  ) \Vert_2^2 \\
= &  \Vert \Lambda_{\tilde{\varphi}}(  \int_G \mu( a^\sharp_{n,t} \ast a_{n,t} ) dt ) - \Lambda_{\tilde{\varphi}}( \mu(x^\sharp \ast x)  ) \Vert_2^2 \\
= & \Vert \Lambda_{\tilde{\varphi}} (\mu(b_n - b)) \Vert_2^2 \\
= & \tilde{\varphi}\left( \mu(b_n - b)^\ast \mu(b_n - b)  \right) \\
= & \varphi\left( ((b_n - b)^\sharp \ast (b_n - b))(e) \right) \\
= & \varphi\left(  \int_G \left| \int_G f_n(ut) f_n(t) dt  -1  \right|^2        (x^\sharp \ast x)(u)^\ast  (x^\sharp \ast x)(u)  du \right) \\
\rightarrow &\: 0.
\end{split}
\]
The limit $n \rightarrow \infty$ in the last line is justified by the Fubini Lemma \ref{Lem=Fubini}, the Lebesgue dominated convergence sequence theorem and the definition of the approximating sequences, as
\begin{align*}  \left|\int_G f_n(ut) f_n(t) dt  -1 \right| & = \left|\int_G (f_n(ut)-f_n(t)) f_n(t) dt  \right| \leq \left(\int_G |f_n(t)|^2 dt\right)^{\frac{1}{2}}
 \left(\int_G |f_n(t)-f_n(ut)|^2 dt\right)^{\frac{1}{2}} \\
 &\leq \left(\int_G |f_n(t)^2 -f_n(ut)^2 |dt\right)^{\frac{1}{2}} = \Vert g_n - \delta_{u^{-1}} \ast g_n\|_1^{\frac{1}{2}}.\end{align*}

\end{proof}
%%%%%%%%%%%%%%%%%%%%%%%%%%%%%
%START OF THE BLUE COLOR PART
%%%%%%%%%%%%%%%%%%%%%%%%%%%%%
{

We may now conclude the dual of Theorem \ref{Thm=HPCrossedProduct}. Its proof follows -- of course -- along the same lines. The main conceptual difference lies in the fact that we require amenability of $G$.

\begin{thm}\label{Thm=HPCrossedProductDual}
Let $\cM$ be a von Neumann algebra. Let $G$ be a locally compact, amenable group. Let $\alpha: G \rightarrow \Aut(\cM)$ be a strongly continuous   group of automorphisms. Let $\cN := \cM \rtimes_{\alpha} G$ be the crossed product. If $\cM$ has the Haagerup property, then so does $\cN$.
\end{thm}
\begin{proof}
Without loss of generality, we may assume that there exists a normal, semi-finite, faithful weight $\varphi$ on $\cM$ for which the action of $\alpha$ is $\varphi$-preserving. Indeed, $\cM$ has the Haagerup property if and only if $\cM \wot \cB(L^2G)$ has the Haagerup property and $\cN$ has the Haagerup property if and only if $\cN \wot \cB(L^2 G)$ has the Haagerup property, see Proposition \ref{Prop=HPBofH}, Proposition \ref{Lem=DescentDownToTensor} and Lemma \ref{Lem=HPTensor}.  We have $\cN \wot\cB(L^2 G) = (\cM \wot \cB(L^2 G)) \rtimes_\beta G$ (again see \cite{EnoSch}), so it suffices to prove the theorem for the case $\alpha := \beta$ (and hence with $\cM$ and $\cN$ replaced with their tensor products with $\cB(L^2 G)$). But $\cM \wot \cB(L^2 G)$ has several $\beta$-invariant normal, semi-finite, faithful weights -- in fact every weight $\psi \circ S$ with $\psi$ a normal, semi-finite, faithful weight on $\cN$ provides such an example (see Subsection \ref{Subsection:crossprel} for the definition of $S$).

Since $G$ is amenable, it admits an approximating sequence $(g_n)_{n=1}^{\infty}$ in $C_c(G)$. Define for each $n \in \bn$, as before, the function  $f_n = g_n^{\frac{1}{2}}$, set $F_n = \textup{supp}\, f_n$ and the projection $e_n = \nu(\chi_{F_n})$.  Define the normal, semi-finite weight $\varphi_n$ on $\cM \wot \cB(L^2 G)$, c.f.\ Lemma \ref{Lem=HaagerupKrausDual}, by the formula
\[
\varphi_n(x) =  \tilde{\varphi} \circ S_{f_n}(x) = \tilde{\varphi} \circ S(\nu(f_n) x \nu(f_n)), \qquad  x \in (\cM \wot \cB(L^2 G))^+,
\]
and consider its restriction to the algebra $e_n (\cM \wot \cB(L^2 G)) e_n$. We observe that it remains then a faithful, normal semifinite weight: faithfulness can be shown (using the fact that the operator valued weight $S$ and the weight $\tilde{\varphi}$ are faithful) as in the beginning of the proof of Proposition \ref{Prop=CrossedProductEmbeddingUltra}, normality is clear (as $S$ and $\tilde{\varphi}$ are normal) and finally semi-finiteness follows from Lemma \ref{Lem=SemiFiniteCornerII}.
Using  Proposition \ref{Prop=HPCornerNonModular}, we then see that $(e_n (\cM \wot \cB(L^2 G)) e_n, \varphi_n)$ has the Haagerup property.

\vspace{0.3cm}

\noindent {\bf Step 0.} Now, let $\{ \Phi_k^{(n)} \}_{k \in \mathbb{N}}$ be the sequence of completely positive maps witnessing the Haagerup property of $(e_n (\cM \otimes \cB(L^2 G)) e_n, \varphi_n)$. That is, $\varphi_n \circ \Phi_k^{(n)} \leq \varphi_n$, the $L^2$-identification of the map $\Phi_k^{(n)}$, denoted by $T_k^{(n)}$, is compact and  $T_k^{(n)} \rightarrow 1$ strongly as $k \rightarrow \infty$. By the proof of Proposition \ref{Prop=HPCornerNonModular} we may in fact assume that $T_k^{(n)}$ is uniformly bounded both in $n$ and $k$.
Recall that naturally $\cN \subseteq \cM \wot \mathcal{B}(L^2 G)$. Using Lemma \ref{Lem=HaagerupKrausDual}, define the normal, completely positive map $\cN \rightarrow \cN$ by,
\[
\Psi_k^{(n)}(x) =  S_{f_n} \left( \Phi_k^{(n)}(  e_n x e_n )   \right), \qquad x \in \cN.
\]

\vspace{0.3cm}

\noindent {\bf Step 1.} We have the following estimate. We use respectively the definition of $\Psi_k^{(n)}$, the defining property of $\Phi_k^{(n)}$, an elementary equation and Lemma \ref{Lem=HaagerupKrausDual} (which uses $\alpha$-invariance of $\varphi$). For $y \in \cN^+$  we have,
\begin{equation}\label{Eqn=StatePreservingDual}
\begin{split}
\tilde{\varphi} ( \Psi_k^{(n)}(y)  ) = &  \tilde{\varphi} \circ S_{f_n} \left( \Phi_k^{(n)}(  e_n y e_n )   \right)
\leq   \tilde{\varphi} \circ S_{f_n} \left(   e_n y e_n     \right)  \\
 = & \tilde{\varphi} \circ S_{f_n} \left(     y       \right)
  \leq \tilde{\varphi}(y).
\end{split}
\end{equation}
This shows that $\Psi_k^{(n)}$ satisfies Definition \ref{Dfn=HP} \ref{Item=HPOne} for $(\cN, \tilde{\varphi})$ for every $k,n \in \mathbb{N}$.

\vspace{0.3cm}

\noindent {\bf Step 2.} Next, let us   prove that $\Psi_k^{(n)}$ also satisfies Definition \ref{Dfn=HP} \ref{Item=HPTwo} for every $k,n \in \mathbb{N}$. That is, we shall show that,
\[
S_k^{(n)}: \Lambda_{\tilde{\varphi}}(x) \mapsto \Lambda_{\tilde{\varphi}}(\Psi_k^{(n)}(x)), \qquad \forall \: x \in \mathfrak{n}_{\tilde{\varphi}},
\]
determines a compact operator on $L^2(\cN, \tilde{\varphi})$, with a norm that is uniformly bounded in $k,n \in \mathbb{N}$.
% Moreover, there exist $k,l \in \mathbb{N}$ such that,
%\[
%\Vert \Lambda_{\varphi}(x) - \Lambda_{\varphi}(\Psi_k^{(l)}(x)) \Vert_2 \leq \frac{1}{\vert F \vert}, \qquad \forall \: x \in F.
%\]
%Here, $\vert F \vert$ denotes the number of elements in the set $F$. This then concludes the proof of the theorem.
In order to do so, we need the following maps.
\begin{itemize}
\item Firstly, by Proposition \ref{Prop=CrossedProductEmbeddingUltraDual} there exists a contraction,
\[
V_n: L^2(\cN, \tilde{\varphi}) \rightarrow L^2(e_n(\cM \otimes \cB(L^2 G))e_n, \varphi_n): \Lambda_{\tilde{\varphi}}(x) \mapsto \Lambda_{\varphi_n}(e_n  x e_n).
\]
\item Secondly, it follows from Lemma \ref{Lem=HaagerupKrausDual} and Lemma \ref{Lem=KadisonSchwarz} that there exists a contractive map,
    \[
      A_n: L^2( e_n (\cM \otimes \cB(L^2 G) ) e_n, \varphi_n) \rightarrow L^2(\cN, \tilde{\varphi}): \Lambda_{\varphi_n}(y) \mapsto \Lambda_{\tilde{\varphi}}(S_{f_n}(y)).
    \]

\end{itemize}

    %Firstly, note that Lemma \ref{Lem=HaagerupKrausComputation} shows that for $x \in \nphi$, we have,
%\[
%\begin{split}
%\Vert \Lambda_{\tilde{\varphi}_i}(x) \Vert_2^2 = & \tilde{\varphi}(\lambda(f_i)^\ast x^\ast x \lambda(f_i) ) = \varphi( T( \lambda(f_i)^\ast x^\ast x \lambda(f_i) )) \\
%= & \varphi( \int_{\mathbb{R}} \vert f_i(s) \vert^2  \alpha_{-s}(x^\ast x) ds ) = \varphi(x^\ast x) = \Vert \Lambda_{\varphi}(x) \Vert_2^2.
%\end{split}
%\]
%Therefore, the GNS-map of $\varphi$ may be related to the GNS-map of $\tilde{\varphi}$ as follows. For every $i$, there exists an isometric map:

Now, we claim that we have the following intertwining property:
\begin{equation}\label{Eqn=IntertwiningPropertyDual}
S_{k}^{(n)} =  A_n T_k^{(n)} V_n.
\end{equation}
Since $T_k^{(n)}$ is compact, this then implies that $S_k^{(n)}$ is compact. To prove the claim, let $y = \mu(x) \in \cN$ with $x \in B_\varphi$, see Lemma \ref{Lem=CoreLemma}. Then,
\[
\begin{split}
&  A_n T_k^{(n)} V_n \Lambda_{\tilde{\varphi}}(y) \\
=& A_n T_k^{(n)} \Lambda_{\varphi_n}(e_n y e_n) \\
= & A_n \Lambda_{\varphi_n}( \Phi_k^{(n)}(e_n y e_n) ) \\
 =&  \Lambda_{\tilde{\varphi}}  (  S_{f_n}( \Phi_k^{(n)}(e_n y e_n) ))  \\
 = & \Lambda_{\tilde{\varphi}}( \Psi_k^{(n)}(y) ) \\
= & S_k^{(n)} \Lambda_{\tilde{\varphi}}(y).
\end{split}
\]
The claim follows then by Lemma \ref{Lem=CoreLemma}.

\vspace{0.3cm}

\noindent {\bf Step 3.} We now arrive at choosing $k, n  \in \mathbb{N}$ properly in order to construct a sequence of maps $\Psi_k^{(n)}$.   For the remainder of the proof we fix  a finite subset $F \subseteq \mu(B_\varphi)$. In particular, $F \subseteq \mathfrak{n}_{\tilde{\varphi}}$. By Lemma \ref{Lem=SFLStrongToOne} we see  that there exists an $n \in \mathbb{N}$ such that for every $x \in F$ we have,
\[
\begin{split}
& \Vert A_n \Lambda_{\varphi_n}(e_n x e_n) - \Lambda_{\tilde{\varphi}}(x) \Vert_2 \\
= & \Vert \Lambda_{\tilde{\varphi}}( S_{f_n}(x) ) - \Lambda_{\tilde{\varphi}}(x)    \Vert_ 2 \leq \frac{1}{2 \vert F \vert}.
\end{split}
\]
 Next, we may choose $k \in \mathbb{N}$ such that for every $x \in F$ we have
\[
\Vert T_k^{(n)} \Lambda_{\varphi_n}( e_n x  e_n) - \Lambda_{\varphi_n}(e_n x  e_n) \Vert_2 \leq \frac{1}{2 \vert F \vert},
\]
by definition of $T_k^{(n)}$. In that case, for every $x \in F$, we have,
\[
\begin{split}
& \Vert A_n T_k^{(n)} \Lambda_{\varphi_n}(e_n x e_n) -  \Lambda_{\tilde{\varphi}}(x) \Vert_2\\
 \leq &
\Vert A_n T_k^{(n)} \Lambda_{ \varphi_n}(e_n x  e_n) - A_n \Lambda_{\varphi_n}(e_n x  e_n) \Vert_2 + \Vert A_n \Lambda_{ \varphi_n}(e_n x  e_n) -  \Lambda_{\tilde{\varphi}}( x) \Vert_2 \leq \frac{1}{\vert F \vert}.
\end{split}
\]
And hence, also for all $x \in F$,
\[
\Vert  A_n T_k^{(n)} V_n \Lambda_{\tilde{\varphi}}(x) - \Lambda_{\tilde{\varphi}}(x) \Vert \leq \frac{1}{\vert F\vert}.
\]
This concludes the proof of the theorem.
\end{proof}

%%%%%%%%%%%%%%%%%%%%%%%%%%%%%%%%%%%%%
%END OF THE BLUE COLOR PART
%%%%%%%%%%%%%%%%%%%%%%%%%%%%%%%%%%%%%
}

%\appendix

%\section{}

\vspace{0.3cm}

\subsection*{Acknowledgements} AS acknowledges useful discussions with P. Fima on the topics studied in this article. The authors thank N. Ozawa and the anonymous referee for pointing out a mistake in an earlier version of this paper. The referee's comments have also led to some simplifications and clarifications in our proofs.

\end{document}